\documentclass[11pt]{amsart}
\usepackage{amsmath}
\usepackage{amssymb}
\usepackage{stmaryrd}
\usepackage{amsthm}
\usepackage{amscd}
\usepackage{amsfonts}
\usepackage{graphicx}
\usepackage{fancyhdr}
\usepackage[utf8]{inputenc}
\usepackage{mathrsfs}
\usepackage{tikz-cd}
\usepackage{mathtools}
\usepackage[all,cmtip]{xy}
\usepackage{hyperref}
\numberwithin{equation}{section}
\newtheorem{theorem}[equation]{Theorem}

\newtheorem{assumption}[equation]{Assumption}
\newtheorem{lemma}[equation]{Lemma}
\newtheorem{proposition}[equation]{Proposition}
\newtheorem{definition}[equation]{Definition}

\newtheorem*{corollary*}{Corollary}

\theoremstyle{definition}
\newtheorem{remark}[equation]{Remark}

\def\Gal{\mathrm{Gal}}
\def\GL{\mathrm{GL}}
\def\PGL{\mathrm{PGL}}
\def\GSp{\mathrm{GSp}}
\def\PGSp{\mathrm{PGSp}}

\def\Sp{\mathrm{Sp}}

\def\GU{\mathrm{GU}}
\def\GO{\mathrm{GO}}

\def\Spec{\mathrm{Spec}}

\def\min{\mathrm{min}}

\def\CNL{\mathrm{CNL}}

\def\Hom{\mathrm{Hom}}
\def\End{\mathrm{End}}
\def\ART{\mathrm{ART}}
\def\lr{\mathrm{lr}}
\def\min{\mathrm{min}}
\def\mix{\mathrm{mix}}
\def\Iw{\mathrm{Iw}}
\def\Spf{\mathrm{Spf}}
\def\univ{\mathrm{univ}}

\def\Fil{\mathrm{Fil}}

\def\Res{\mathrm{Res}}

\def\calD{\mathcal{D}}

\def\calF{\mathcal{F}}
\def\calG{\mathcal{G}}

\def\calL{\mathcal{L}}

\def\calM{\mathcal{M}}
\def\calO{\mathcal{O}}

\def\calS{\mathcal{S}}
\def\calT{\mathcal{T}}

\def\fraca{\mathfrak{a}}
\def\fracm{\mathfrak{m}}
\def\fracn{\mathfrak{n}}

\def\CC{\mathbb{C}}

\def\FF{\mathbb{F}}
\def\GG{\mathbb{G}}
\def\bfG{\mathbf{G}}

\def\QQ{\mathbb{Q}}

\def\TT{\mathbb{T}}
\def\bfT{\mathbf{T}}

\def\ZZ{\mathbb{Z}}

\def\rma{\mathrm{a}}
\def\rmb{\mathrm{b}}
\def\rmc{\mathrm{c}}

\def\rmu{\mathrm{u}}

\def\rmA{\mathrm{A}}
\def\rmB{\mathrm{B}}

\def\rmF{\mathrm{F}}
\def\rmG{\mathrm{G}}
\def\rmH{\mathrm{H}}
\def\rmI{\mathrm{I}}
\def\rmJ{\mathrm{J}}
\def\rmK{\mathrm{K}}

\def\rmM{\mathrm{M}}
\def\rmN{\mathrm{N}}
\def\rmP{\mathrm{P}}
\def\rmQ{\mathrm{Q}}

\def\rmW{\mathrm{W}}
\def\rmR{\mathrm{R}}
\def\rmS{\mathrm{S}}
\def\rmX{\mathrm{X}}
\def\rmY{\mathrm{Y}}
\def\rmT{\mathrm{T}}
\def\rmZ{\mathrm{Z}}

\usepackage[top=1.5in, bottom=1.5in, left=1.3in, right=1.3in]{geometry}
\begin{document}
\title[Deformation for rigid Galois representations]
{Deformation of rigid Galois representations and cohomology of certain quaternionic unitary Shimura variety}

\begin{abstract}
In this article, we use deformation theory of Galois representations valued in the symplectic group of degree four to prove a freeness result for the cohomology of certain quaternionic unitary Shimura variety over the universal deformation ring for certain type of residual representation satisfying a property called rigidity. This result plays an important role in the proof of the arithmetic level raising theorem for the symplectic similitude group of degree four over the field of rational numbers by the author. 
\end{abstract}

\author{Haining Wang }
\address{\parbox{\linewidth}{Haining Wang\\Shanghai Center for Mathematical Sciences,\\ Fudan University,\\No.2005 Songhu Road,\\Shanghai, 200438, China.~ }}
\email{wanghaining1121@outlook.com}

\maketitle
\tableofcontents
\section{Introduction}
Let $\rmF$ be a totally real field. In this article, we study deformations of Galois representations of $\rmF$ valued in $\GSp_{4}$. The main purpose of this article is to deduce a freeness result for the cohomology of certain quaternionic unitary Shimura variety using the Taylor-Wiles patching method. This quaternionic unitary Shimura variety is closely related to the classical Hilbert-Siegel modular variety. Their relation is similar to that of classical modular curves with Shimura curves associated to indefinite quaternion algebras. The deformation problem we imposed on the residual representation is tailored to have applications in proving the arithmetic level raising theorem in \cite{Wang}. In particular, following the terminology of \cite{LTXZZa}, we show that if the residual Galois representation satisfies certain property called rigidity, then we can prove an $\rmR=\rmT$ type theorem as well as the disired freeness result for the cohomology of certain quaternionic unitary Shimura varieties. This is our first result. Let $\pi$ be a cuspidal automorphic representation of $\GSp_{4}(\mathbb{A}_{\rmF})$ of general type. Then one can associate to it a compatible family of Galois representations of $\rmF$ valued in $\GSp_{4}$. Suppose that the residual Galois representation satisfies a large image assumption as well as  suitable hypothesis, we can show that these residual Galois representations are indeed rigid if the residual characteristic of the coefficient field is large enough. This is our second main result. 

We now state our two main results more formally in a combined fashion. Let $\Pi$ be a cuspidal automorphic representation of $\GL_{4}(\mathbb{A}_{\rmF})$, then we say that $\Pi$ is of symplectic type with similitude character $\psi$ if the partial $L$-function $L^{\rmS}(s, \Pi, \wedge^{2}\otimes \psi^{-1})$ has a pole at $s=1$ where $\rmS$ is any finite set of places of $\rmF$. Following Arthur \cite{}, we say a cuspidal automorphic representation $\pi$ is of general type, if there is a cuspidal automorphic representation $\Pi$ of $\GL_{4}(\mathbb{A}_{\rmF})$ such that $\Pi$ is of symplectic type with similitude character $\psi$ and that for each place $v$ of $\rmF$, the $L$-parameter $\mathrm{rec}_{\GL_{4}}(\Pi_{v})$ is obtained from composing $\mathrm{rec}_{\mathrm{GT}}(\pi_{v})$ with the embedding $\GSp_{4}\hookrightarrow \GL_{4}$. In this case $\Pi $ is called a transfer of $\pi$. We will only consider those $\pi$ of general type in this article. Suppose we are given the following datum
\begin{itemize}
\item $l\geq 3$ and an isomorphism $\iota_{l}:\CC\cong\overline{\QQ}_{l}$;
\item a cuspidal automorphic representation $\pi$ of general type with weight $(k_{v}, l_{v})_{v\mid \infty}$ and trivial central character; This means that the Harish-Chandra parameter of $\pi$ is given by $(k_{v}-1, l_{v}-2, 0)_{v\mid\infty}$;
\item a number field $E\subset \CC$ called strong coefficient field of $\pi$, see Definition \ref{strong-field}, such that one has a  Galois representation
\begin{equation*}
\rho_{\pi,\lambda}:\rmG_{\rmF}\rightarrow \GSp_{4}(E_{\lambda})
\end{equation*}
attached to $\pi$, where $\lambda$ is a prime of $E$ induced by the isomorphism $\iota_{l}:\CC\cong\overline{\QQ}_{l}$. We will write $\calO$ as the valuation ring of the local field $E_{\lambda}$. We assume that  the similitude character of $\rho_{\pi,\lambda}$ is given by $\epsilon^{-3}_{l}$; 
\item  a finite set $\Sigma$ of non-archimedean places of $\rmF$ such that for every non-archimedean place $v$ of $\rmF$ not in $\Sigma$, $\pi_{v}$ is unramified.
\end{itemize}

We will assume that $\rho_{\pi,\lambda}$ is residually irreducible, then we have the the residual representation $\overline{\rho}_{\pi,\lambda}$ well-defined up to isomorphism. We consider the deformation problem for $\overline{\rho}_{\pi,\lambda}$ such that 
\begin{itemize}
\item $\Sigma$ admits a partition $\Sigma=\Sigma_{\lr}\cup\Sigma_{\min}\cup\Sigma_{l}$;
\item for places in $\Sigma_{\min}$, we classify the minimal deformations;
\item for places in $\Sigma_{\lr}$, $\pi_{v}$ is of type $\rmI\rmI\rma$ in the classification of Sally-Tadic \cite{ST-classification} and Schmidt \cite{Sch-Iwahori}, we classify deformations with certain prescribed ramifications;
\item for places in $\Sigma_{l}$, we classify deformations that are regular Fontaine-Laffaille crystalline.
\end{itemize}
Then we obtain a universal deformation ring $\rmR^{\univ}_{\calS}$ for this deformation problem. Let $\TT^{\Sigma}$ be the abstract Hecke algebra of $\GSp_{4}$ unramified away from $\Sigma$. Then one can associate to $\pi$ and the strong coefficient field $E$, a homomorphism $\phi_{\pi,\lambda}: \TT^{\Sigma}\rightarrow \calO$. Then we say the ideal $\fracm=\ker(\TT^{\Sigma}\xrightarrow{\phi_{\pi,\lambda}} \calO\rightarrow \calO/\lambda)$ is the maximal ideal corresponding to the residual Galois representation $\overline{\rho}_{\pi, \lambda}$. Let
$\rmB$ be a quaternion algebra over $\rmF$ and we consider two cases.
\begin{itemize}
\item The quaternion algebra $\rmB$ is totally definite that is $\rmB$ is ramified at all infinite places $\Sigma_{\infty}$. We will refer to this case as the definite case.  Let $d(\rmB)=0$ in this case.
\item The quaternion algebra $\rmB$ is ramified at all infinite places $\Sigma_{\infty}$ except for one infinite place $v^{\mathrm{split}}_{\infty}$.  We will refer to this case as the indefinite case.  Let $d(\rmB)=3$ in this case.

\item Let $\Sigma_{\rmB}$ be the places of $\rmF$ such that the quaternion algebra $\rmB$ is ramified. The we will always assume that $\Sigma_{\rmB}$ is contained in $\Sigma_{\lr}$.
\end{itemize}
We have the quaternionic unitary group of degree two over $\rmF$ denoted by $\GU_{2}(\rmB)$ and we define $\bfG(\rmB)=\Res_{\rmF/\QQ}\GU_{2}(\rmB)$. We can associate the following datum to $\bfG(\rmB)$:
\begin{itemize}
\item $\xi\in (\ZZ^{3}_{\geq 0})^{\Sigma_{\infty}}$ gives a highest weight of $\bfG(\rmB)(\CC)\cong \prod\limits_{v\mid \infty}\GSp_{4}(\CC)$; We assume that $\xi=(\xi_{1, v}, \xi_{2, v}, 0)_{v\mid \infty}$ with $\xi_{1, v}\geq\xi_{2, v}\geq 0$ and $\xi_{1, v}\equiv \xi_{2, v} \mod 2$ such that $\xi$ is compatible with the weight of $\pi$ in the sense that $\xi=(k_{v}-3, l_{v}-3, 0)_{v\mid \infty}$.
\item  a neat open compact subgroup $\rmK$ of $\GU_{2}(\rmB)(\mathbb{A}_{\rmF})$ of the form 
\begin{equation*}
\rmK=\prod\limits_{v\in\Sigma_{\min}\cup\Sigma_{\lr}}\rmK_{v}\times\prod\limits_{v\not\in\Sigma_{\min}\cup\Sigma_{\lr}}\GSp_{4}(\calO_{\rmF_{v}})
\end{equation*}
where $\rmK_{v}$ is the paramodular subgroup of $\GU_{2}(\rmB_{v})$ or $\GSp_{4}(\rmF_{v})$ for $v\in \Sigma_{\lr}$, $\rmK_{v}$ is contained in the pro-$v$ Iwahori subgroup of $\GSp_{4}(\rmF_{v})$ for $v\in \Sigma_{\min}$.
\end{itemize}

This datum defines a Shimura variety resp. a Shimura set $\mathrm{Sh}(\rmB, \rmK)$ in the indefinite case resp. in the definite case. And we also have an $\calO$-local system $\calL_{\xi}$ on $\mathrm{Sh}(\rmB, \rmK)$.  The following is our main result.

\begin{theorem}
Let $\pi$ be a cuspidal automorphic representation of $\GSp_{4}(\mathbb{A}_{\rmF})$ of general type with trivial central character as above, suppose the following assumptions hold.
\begin{itemize}
\item[(D0)] $l\geq 3$ is a prime unramified in $\rmF$. 
\item[(D1)] The weight $(k_{v}, l_{v})_{v\mid\infty}$ of $\pi$ satisfies $l_{v}+k_{v}\leq l+1$;
\item[(D2)] The image $\overline{\rho}_{\pi,\lambda}(\rmG_{\rmF})$ contains $\GSp_{4}(\FF_{l})$;
\item[(D3)] $\overline{\rho}_{\pi, \lambda}$ is rigid for $(\Sigma_{\mathrm{min}}, \Sigma_{\mathrm{lr}})$;
\item[(D4)] Suppose that $\rmB$ is indefinite. For every finite set $\Sigma^{\prime}$ of nonarchimedean places of $\rmF$ containing $\Sigma$, and every open compact subgroup $\rmK^{\prime}\subset \rmK$ satisfying $\rmK^{\prime}_{v}=\rmK_{v}$ for $v\not\in \Sigma^{\prime}$, we have 
\begin{equation*}
\rmH^{d}(\mathrm{Sh}(\rmB, \rmK^{\prime}), \calL_{\xi}\otimes_{\calO}k)_{\fracm^{\prime}}=0
\end{equation*}
for every $d\neq 3$ and $\fracm^{\prime}=\fracm\cap \TT^{\Sigma^{\prime}}$.
\end{itemize}
Let $\mathbf{T}$ be the image of $\TT^{\Sigma}$ in $\mathrm{End}_{\calO}(\rmH^{d(\rmB)}(\mathrm{Sh}(\rmB, \rmK), \calL_{\xi}))$ and suppose that $\mathbf{T}_{\fracm}$ is non-zero. Then we have the following conclusions.
\begin{enumerate}
\item There is an isomorphism of complete intersection rings: $\rmR^{\univ}_{\calS}\cong\mathbf{T}_{\fracm}$.
\item The localized cohomology $\rmH^{d(\rmB)}(\mathrm{Sh}(\rmB, \rmK), \calL_{\xi})_{\fracm}$ is finite free module over $\mathbf{T}_{\fracm}$ and hence $\rmH^{d(\rmB)}(\mathrm{Sh}(\rmB, \rmK), \calL_{\xi})_{\fracm}$ is also finite free over $\rmR^{\univ}_{\calS}$.
\end{enumerate}
Suppose that  the image $\overline{\rho}_{\pi,\lambda}(\rmG_{\rmF})$ contains $\GSp_{4}(\FF_{l})$ for sufficiently large $l$, then the representation $\overline{\rho}_{\pi, \lambda}$ is indeed rigid for sufficiently large $l$.
\end{theorem}

\subsection{Notation}
We will use common notations and conventions in algebraic number theory and algebraic geometry. The cohomology  of schemes appear in this article will be understood as computed over the \'{e}tale sites. 

For a field $\rmK$, we denote by $\overline{\rmK}$ a separable closure of $\rmK$ and put $\rmG_{\rmK}:=\Gal(\overline{\rmK}/\rmK)$ the Galois group of $\rmK$. Suppose $\rmK$ is a number field and $v$ is a place of $\rmK$, then $\rmK_{v}$ is the completion of $\rmK$ at $v$ with valuation ring $\calO_{v}$ whose maximal ideal is also written as $v$, if no danger of confusion could arise, we also denote by $v$ a fixed uniformizer of $v$. We let $\mathbb{A}_{\rmK}$ be the ring of ad\`{e}les over $\rmK$ and $\mathbb{A}^{(\infty)}_{\rmK}$ be the subring of finite ad\`{e}les.  
Let $\rmK$ be a number field. Let $\Sigma_{\infty}$ be the set of infinity places of $\rmK$.  For each rational prime $p$, let $\Sigma_{p}$ be the set of places of $\mathrm{K}$ over $p$ and $\Sigma_{\mathrm{bad}}$ be the set of ramified places. 

For our convention, let $\GSp_{4}$ be the reductive group over $\ZZ$ defined by 
$\GSp_{4}=\{g\in\GL_{4}: g\rmJ g^{t}=\nu(g)\rmJ\}$ where $\rmJ$ is the antisymmetric matrix given by $\begin{pmatrix}0&s\\-s&0\\ \end{pmatrix}$
for $s=\begin{pmatrix}0&1\\ 1&0 \\ \end{pmatrix}$ and where $\nu$ is the similitude character of $\GSp_{4}$. By definition, we have an exact sequence 
$0\rightarrow \Sp_{4}\rightarrow \GSp_{4}\xrightarrow{\nu}\GG_{m}\rightarrow 0$.

\section{Generalities on deformation theory}

\subsection{Galois deformation rings} We summarizes some more or less familiar notions and results of deformation theory of Galois representations valued in $\GSp_{4}$ over a totally field $\rmF$. Our writing is influenced by the exposition of \cite[\S 7]{BCGP}.
\begin{itemize}

\item Let $\CNL_{\calO}$ be the category of complete Noetherian local $\calO$-algebras with residue field $k$. An object of $\CNL_{\calO}$ is called a $\CNL_{\calO}$-algebra. Let $\ART_{\calO}$ be the category of Artinian local rings over $\calO$ with residue field $k$.

\item We fix $\overline{\rho}: G_{\rmF}\rightarrow \GSp_{4}(k)$ a continuous absolutely irreducible homomorphism and a character $\psi: G_{\rmF}\rightarrow \calO^{\times}$ lifting $\nu\circ \overline{\rho}$.

\item Let $\rmS$ be finite set containing all the places over $l$ in $\rmF$. Let $\rmF_{\rmS}$ be the maximal unramified extension of $\rmF$ outside $\rmS$. Let $\rmG_{\rmF, \rmS}$ be the Galois group $\Gal(\rmF_{S}/\rmF)$.

\item A lift of $\overline{\rho}_{v}=\overline{\rho}\vert_{G_{\rmF_{v}}}$ is a continuous homomorphism $\rho_{v}: G_{\rmF_{v}}\rightarrow \GSp_{4}(A)$ to a $\CNL_{\calO}$-algebra $A$, such that $\rho_{v} \mod \fracm_{A}$ agrees with $\overline{\rho}_{v}$ and $\nu\circ\rho_{v}=\psi\vert_{G_{\rmF_{v}}}$.
\item Let $\calD^{\square}_{v}$ denote the set valued functor on $\CNL_{\calO}$ that sends an $A$ to the set of lifts of $\overline{\rho}_{v}$ to $A$. This functor is representable by a ring denoted by $\mathrm{R}^{\square}_{v}$.
\end{itemize}
\begin{definition}
A local deformation problem for $\overline{\rho}_{v}$ is subfunctor $\calD_{v}$ of $\calD^{\square}_{v}$ satisfying the following conditions:
\begin{itemize}
\item $\calD_{v}$ is representable by a quotient $\rmR_{v}$ of $\rmR^{\square}_{v}$;
\item For all $A\in\CNL_{\calO}$, $\rho_{v}\in \calD_{v}(A)$ and $a\in \ker(\GSp_{4}(A)\rightarrow \GSp_{4}(k))$, we have $a\rho_{v} a^{-1}\in \calD_{v}(A)$.
\end{itemize}
\end{definition}
\begin{definition}
A global deformation problem is a tuple $\calS=(\overline{\rho}, \rmS, \psi, \{\calD_{v}\})$ where
\begin{itemize}
\item $\overline{\rho}: G_{\rmF}\rightarrow \GSp_{4}(k)$, $\rmS$ and $\psi$ is as above;
\item For each $v\in \rmS$, $\calD_{v}$ is a local deformation problem for $\rho_{v}$.
\end{itemize}
\end{definition}

Suppose we are given a global deformation problem $\calS=(\overline{\rho}, \rmS, \psi, \{\calD_{v}\})$. Then we recall the following notions and facts.
\begin{itemize}
\item  A \emph{lift} of $\overline{\rho}$ is a continuous homomorphism $\rho: G_{\rmF}\rightarrow \GSp_{4}(A)$ for a $\CNL_{\calO}$-algebra $A$, such that $\rho \mod \fracm_{A}$ agrees with $\overline{\rho}$ and $\nu\circ\rho=\psi$.
\item two lifts $\rho_{1}, \rho_{2}: G_{\rmF}\rightarrow \GSp_{4}(A)$ are strictly equivalent if there is an 
\begin{equation*}
a\in \ker(\GSp_{4}(A)\rightarrow \GSp_{4}(k)) 
\end{equation*}
such that $\rho_{2}=a\rho_{1}a^{-1}$. A \emph{deformation} of $\overline{\rho}$ is a strict equivalence class of lifts of $\overline{\rho}$.
\item We say $\rho: G_{\rmF}\rightarrow \GSp_{4}(A)$ is a \emph{lift of type $\calS$} if $\rho\vert_{G_{\rmF_{v}}}\in\calD_{v}(A)$ for each $v\in \rmS$. A \emph{deformation of $\overline{\rho}$} is a strict equivalence class of lifts of $\overline{\rho}$. 
\item We denote by $\calD_{\calS}$ the set valued functor sending $A\in\CNL_{\calO}$ to the set of lifts $\rho: G_{\rmF}\rightarrow \GSp_{4}(A)$ of $\overline{\rho}$ of type $\calS$. The functor is represented by $\rmR_{\calS}$.

\item Let $\rmT\subset \rmS$, a $\rmT$-framed lift of type $\calS$ is tuple $(\rho, \{\gamma_{v}\}_{v\in\rmT})$ where $\rho$ is a lift of $\overline{\rho}$ of type $\calS$, and $\gamma_{v}\in\ker(\GSp_{4}(A)\rightarrow\GSp_{4}(k))$ for each $v\in \rmT$. Let $(\rho_{1}, \{\gamma_{1,v}\}_{v\in\rmT})$  and  $(\rho_{2}, \{\gamma_{2, v}\}_{v\in\rmT})$ be two $\rmT$-framed deformation of type $\calS$. They are called \emph{strictly equivalent} if $\rho_{2}=a\rho_{1}a^{-1}$ and $\gamma_{2,v}=a\gamma_{1, v}$ for each $v\in \rmT$. 

\item A strict equivalence class of $\rmT$-framed lifts of type $\calS$ is called a $\rmT$-framed deformation of type $\calS$. The functor sends an algebra $A\in\CNL_{\calO}$ to the set of $\rmT$-framed deformations to $A$ of type $\calS$ is denoted by $\calD^{\rmT}_{\calS}$ and is represented by $\rmR^{\rmT}_{\calS}$. 

\item The natural transformation from $\calD^{\rmT}_{\calS}$ to $\calD_{\calS}$ by forgetting the $\rmT$-frames is a torsor under $\prod_{v\in \rmT}\widehat{\GSp_{4}}/\widehat{\mathbb{G}}_{m}$. 
Note that the coordinate ring of $\calT=\prod_{v\in \rmT}\widehat{\GSp_{4}}/\widehat{\mathbb{G}}_{m}$ is power series ring over $\calO$ of $11\vert\rmT\vert-1$ variables. Then we have a presentation 
\begin{equation*}
\rmR^{\rmT}_{\calS}\cong \rmR_{\calS}\widehat{\otimes}\calT
\end{equation*}
coming from the splitting of the torsor $\calD^{\rmT}_{\calS}\rightarrow \calD_{\calS}$.
\item Let $\rmT\subset\rmS$ be a subset such that $\calD_{v}=\calD^{\square}_{v}$ for all $v\in\rmS-\rmT$. Let $\rmR^{\rmT, \mathrm{loc}}_{\calS}:=\widehat{\otimes}_{v\in \rmT}\rmR_{v}$. For each $v\in \rmT$, there is a natural transformation $\calD^{\rmT}_{\calS}\rightarrow \calD_{v}$ given by sending $(\rho, \{\gamma_{v}\})$ to $\gamma_{v}\rho\vert_{\rmG_{F_{v}}}\gamma^{-1}_{v}$. This gives rise to a map $\rmR_{v}\rightarrow \rmR^{\rmT}_{\calS}$ and hence a map
\begin{equation*}
\rmR^{\rmT, \mathrm{loc}}_{\calS}\rightarrow \rmR^{\rmT}_{\calS}.
\end{equation*}
\end{itemize}

Consider the Galois cohomology groups 
\begin{equation*}
\begin{aligned}
&\rmH^{i}_{\calS, \rmT}(\rmG_{\rmF, \rmS}, \mathrm{ad}^{0}\overline{\rho}):=\ker(\rmH^{i}(\rmG_{\rmF,{\rmS}}, \mathrm{ad}^{0}\overline{\rho})\rightarrow \prod_{v\in \rmT}\rmH^{i}(\rmG_{\rmF_{v}}, \mathrm{ad}^{0}\overline{\rho}));\\
&\rmH^{i}_{\calS^{\perp}, \rmT}(\rmG_{\rmF, \rmS}, \mathrm{ad}^{0}\overline{\rho}(1)):=\ker(\rmH^{i}(\rmG_{\rmF,{\rmS}}, \mathrm{ad}^{0}\overline{\rho}(1))\rightarrow \prod_{v\in\rmS\backslash \rmT}\rmH^{i}(\rmF_{v}, \mathrm{ad}^{0}\overline{\rho}(1))).\\
\end{aligned}
\end{equation*}
\begin{enumerate}
\item Let $\bullet$ be an element in the set $\{(\calS, \rmT), (\calS^{\perp}, \rmT), \emptyset\}$;
\item Let $\circ$ be any element in the set $\{\rmG_{\rmF, \rmS}, \rmG_{\rmF_{v}}\}$;
\item Let $\ast\in\{0,1\}$. 
\end{enumerate}
Then we define the number $h^{i}_{\bullet}(\circ,\mathrm{ad}^{0}\overline{\rho}(\ast) )$ to be the dimension of $\rmH^{i}_{\bullet}(\circ, \mathrm{ad}^{0}\overline{\rho}(\ast))$. 
\begin{lemma}
Let $\rmT\subset\rmS$ be a subset such that $\calD_{v}=\calD^{\square}_{v}$ for all $v\in\rmS-\rmT$. There is a local $\calO$-algebra surjection 
\begin{equation*}
\rmR^{\rmT,\mathrm{loc}}_{\calS}\lbrack\lbrack \rmX_{1},\cdots, \rmX_{g}\rbrack\rbrack \rightarrow  \rmR^{\rmT}_{\calS}
\end{equation*}
with $g$ given by the formula
\begin{equation*}
\begin{aligned}
g&= h^{1}_{\calS^{\perp}, \rmT}(\rmG_{\rmF, \rmS}, \mathrm{ad}^{0}(1))-h^{0}(\rmG_{\rmF, S}, \mathrm{ad}^{0}\overline{\rho}(1))-\sum_{v\mid \infty}h^{0}(\rmG_{\rmF_{v}}, \mathrm{ad}^{0}\overline{\rho})\\
&+\sum_{v\in\rmS-\rmT}h^{0}(\rmG_{\rmF_{v}}, \mathrm{ad}^{0}\overline{\rho}(1))+\vert \rmT\vert -1.
\end{aligned}
\end{equation*}
\begin{proof}
This is exactly \cite[Proposition 7.2.1]{BCGP}.
\end{proof}

\end{lemma}
\section{local deformation problems}
\subsection{Minimally ramified deformations} Let $q$ be a positive integer co-prime to $l$. We define the $q$-Tame group, denoted by $\rmT_{q}$ to be the semi-direct product $t^{\ZZ_{l}}\rtimes \phi^{\widehat{\ZZ}}_{q}$ where $\phi_{q}$ has the property that $\phi_{q}t\phi^{-1}_{q}=t^{q}$. For every $b$, we can identify $\rmT_{q^{b}}$ as the subgroup of $\rmT_{q}$ topologically generated by $t$ and $\phi_{q^{b}}$. Let $\rmG$ be one of the three classical groups: $\GL_{n}$, $\GSp_{2n}$, $\GO_{n}$ for some natural number $n$. Let $\overline{\rho}:\rmT_{q}\rightarrow G(k)$ be a representation of the $q$-tame group $\rmT_{q}$. Let $\psi: \rmT_{q}\rightarrow \calO^{\times}$ be a character lifting $\nu
\circ\overline{\rho}$. Let $\rmR^{\mathrm{tame}}_{\overline{\rho}}$ be the representing object of the functor $\calD^{\mathrm{tame}}_{\overline{\rho}}$ sending an element $A$ of $\CNL_{\calO}$ to the lifts $\rho: \rmT_{q}\rightarrow \rmG(A)$ of $\overline{\rho}$ satisfying $\nu\circ\rho=\psi$. It follows from \cite[Proposition 3.3.2]{LTXZZa} that $\rmR^{\mathrm{tame}}_{\overline{\rho}}$ is a local complete intersection, flat and pure of relative dimension $10$ over $\calO$. We can define a subfunctor $\calD^{\mathrm{mr}}_{\overline{\rho}}$ of $\calD^{\mathrm{tame}}_{\overline{\rho}}$  that classifies all the minimally ramified liftings following \cite[Defintion 5.4]{Boo19a}. Let $\mathrm{Nil}_{\overline{\rmN}}$ be the functor as in \cite[Defintion 3.9]{Boo19a} which assigns each $A\in\CNL_{\calO}$ the set of pure nilpotents lifting $\overline{\rmN}$ over $A$.

\begin{definition}\label{min-tame}
Let  $\rho: \rmT_{q}\rightarrow \rmG(A)$ be a lifting of  $\overline{\rho}:\rmT_{q}\rightarrow \rmG(k)$ with fixed similitude character $\psi$. We say $\rho$ is minimally ramified if $\rho(t)=\mathrm{exp}(\rmN)$ for some pure nilpotent orbit $\rmN$ in $\mathrm{Nil}_{\overline{\rmN}}(A)$. 
\end{definition}

\begin{lemma}
The deformation problem $\calD^{\mathrm{mr}}_{\overline{\rho}}$ is formally smooth and is represented by $\rmR^{\mathrm{mr}}_{\overline{\rho}}$ which is of relative dimension $10$ over $\calO$.
\end{lemma}
\begin{proof}
This follows from \cite[corollary 5.8]{Boo19a}.
\end{proof}

We also record the following useful \cite[Lemma 3.3.7]{LTXZZ}.

\begin{lemma}\label{decomp}
Let $q^{2}\not\equiv 1\mod l$. Let $\overline{\rho}$ be an unramified representation of $\rmT_{q}=t^{\ZZ_{l}}\rtimes \phi^{\widehat{\ZZ}}_{q}$ on a $k$-vector space $\overline{\rmM}$. Suppose that $\overline{\rmM}$ admits an decomposition $\overline{\rmM}=\oplus^{s}_{i=1}\overline{\rmM}_{i}$ stable under the action of $\overline{\rho}(\phi_{q})$ such that the characteristic polynomials of $\overline{\rho}(\phi_{q})$ on each $\overline{\rmM}_{i}$ are relatively prime to each other. Let $(\rho, \rmM)$ be a lifting of $(\overline{\rho}, \overline{\rmM})$ to $A\in \CNL_{\calO}$. 
\begin{enumerate}
\item There is a unique decomposition $\rmM=\oplus^{s}_{i=1}\rmM_{i}$ of free $A$-modules stable under $\rho(\phi_{q})$ and it is a lifting of $\overline{\rmM}_{i}$ as a $\phi_{q}$-module.
\item Write $\rho(t)=(\rho(t)_{i, j})$ with $\rho(t)_{i, j}\in \Hom_{A}(\rmM_{i}, \rmM_{j})$. Suppose that $q$ is not an eigenvalue for the action of $\phi_{q}$ on $\Hom_{k}(\overline{\rmM}_{i}, \overline{\rmM}_{j})$. Then the decomposition $\rmM=\oplus^{s}_{i=1}\rmM_{i}$ is stable under the action of the whole $\rmT_{q}$.
\end{enumerate}
\end{lemma}
\begin{proof}
This follows from the same proof of \cite[Lemma 3.3.7]{LTXZZ}.
\end{proof}

Let $v$ be a non-archimedean place of $\rmF$. Let $\rmI_{\rmF_{v}}\subset \rmG_{\rmF_{v}}$ be the inertia subgroup at $v$. Let $\rmP_{\rmF_{v}}$ be the maximal subgroup of $\rmI_{\rmF_{v}}$ whose pro-order is prime to $l$. Let $\rmT_{\rmF_{v}}=\rmG_{\rmF_{v}}/\rmP_{\rmF_{v}}$. If the place $v$ has norm $q=q_{v}$, then $\rmT_{\rmF_{v}}$ is a $q$-tame group. 

Let $\overline{\rho}_{v}: \rmG_{\rmF_{v}}\rightarrow \GSp_{4}(k)$ be a representation. The the deformation problem $\calD^{\mathrm{mr}}_{\overline{\rho}_{v}}$ extends to a local deformation problem $\calD^{\mathrm{min}}_{v}$ that classifies minimally ramified lifts of $\overline{\rho}_{v}$. This is done in several steps as we recall now. For an irreducible representation $\theta$ of $\rmP_{\rmF_{v}}$ over $k$, let $\rmG_{\rmF_{v}, \theta}=\{g\in \rmG_{\rmF_{v}}: \theta^{g}\cong\theta\}$ and we have the split short exact sequence which defines $\rmT_{\rmF_{v}, \theta}$ 
\begin{equation*}
0\rightarrow \rmP_{\rmF_{v}}\rightarrow \rmG_{\rmF_{v},\theta}\rightarrow \rmT_{\rmF_{v}, \theta}\rightarrow 0.
\end{equation*}
\begin{lemma}\label{wild-rep}
Let $\theta$ be an irreducible representation of $\rmP_{\rmF_{v}}$ over $k$.
\begin{enumerate}
\item The dimension of $\theta$ is coprime to $l$; 
\item $\theta$ has unique deformation to a representation ${\theta}$ of $\rmP_{\rmF_{v}}$ over $\calO$;
\item $\theta$ as in $(2)$ extends uniquely to a representation $\tilde{\theta}$ of $\rmG_{\rmF_{v}, \theta}\cap \rmI_{\rmF_{v}}$ over $\calO$;
\item $\tilde{\theta}$ as in $(2)$ extends to a representation of $\rmG_{\rmF_{v}, \theta}$ over $\calO$.
\end{enumerate}
\end{lemma}
\begin{proof}
This is given by \cite[Lemma 2.4.11]{CHT}.
\end{proof}
Consider $\rmM_{\theta}(\overline{\rho}_{v})=\Hom_{\rmP_{\rmF_{v}}}(\theta, \overline{\rho}_{v})$ which is a natural $\rmT_{\rmF_{v}, \theta}$-representation. Then $\theta\otimes_{k}\rmM_{\theta}(\overline{\rho}_{v})$ is the $\theta$-isotypic component of $\overline{\rho}_{v}$ which is natural $G_{\rmF_{v},\theta}$-representation. When the underlying $\overline{\rho}_{v}$ is clear, we write $\rmM_{\theta}(\overline{\rho}_{v})$ as $\overline{\rmM}_{\theta}$. Let $\calT=\calT(\overline{\rho}_{v})$ be the isomorphism classes of irreducible representation of $\rmP_{\rmF_{v}}$ such that $\rmM_{\theta}(\overline{\rho}_{v})\neq 0$. We have a conjugation action of $\rmG_{\rmF_{v}}$ on $\calT$ and the quotient is denoted by $\calT/\rmG_{\rmF_{v}}$. Let $[\theta]$ be the orbit of $\theta$ in $\calT/\rmG_{\rmF_{v}}$. Let $\rho_{v}: \rmG_{\rmF_{v}}\rightarrow \GSp_{4}(A)$ be a deformation of $\overline{\rho}_{v}$ over some $A\in \CNL_{\calO}$ and $\rmM$ be the underlying $A$-module of $\rho_{v}$. In fact, we have an isomorphism
\begin{equation*}
\rmM\cong \bigoplus\limits_{[\theta]\in\calT/\rmG_{\rmF_{v}}}\mathrm{Ind}^{\rmG_{\rmF_{v}}}_{\rmG_{\rmF_{v}, \theta}}(\tilde{\theta}\otimes_{A}\rmM_{\theta})
\end{equation*}
where $\rmM_{\theta}$ is the image of $\Hom_{\rmP_{\rmF_{v}}}(\tilde{\theta}, \rmM)$ in $\rmM$ and $[\theta]$ runs through orbits of $\calT/\rmG_{\rmF_{v}}$. Next we consider the symplectic structure on the module $\rmM$ by which we mean $\rmM$ is equipped with a pairing 
\begin{equation*} 
\rmM\times\rmM\rightarrow A(\psi)
\end{equation*}
Let $\rmM^{\vee}=\Hom_{A}(\rmM, A(\psi))$ considered as an symplectic $\rmG_{\rmF_{v}}$-module. Thus we have an isomorphism via its symplectic structure $\rmM\cong \rmM^{\vee}$ and hence and isomorphism 
\begin{equation*}
\bigoplus\limits_{\theta\in\calT}\rmM_{\theta}\cong\bigoplus\limits_{\theta\in\calT}\rmM^{\vee}_{\theta}. 
\end{equation*}
This implies that we have can identify each $\rmM^{\vee}_{\theta}$ with an $\rmM_{\theta^{\ast}}$. This gives rise to a $\rmG_{\rmF_{v}}$-stable partition of $\calT$ into three cases:
\begin{enumerate}
\item $\calT_{1}$ consists of those $\theta$ that are not in the same orbit of $\theta^{\ast}$;
\item $\calT_{2}$ consists of those $\theta$ that are isomorphic to $\theta^{\ast}$;
\item $\calT_{3}$ consists of those $\theta$ that are in the same orbit with $\theta^{\ast}$ but are not isomorphic to $\theta^{\ast}$.
\end{enumerate}
For each $\theta\in\calT_{1}$, we define $\overline{\rmG}_{\theta}=\mathrm{Aut}(\overline{\rmM}_{\theta})$. For each $\theta\in\calT_{2}$, $\theta$ is isomorphic to $\theta^{\ast}$ and hence $\overline{\rmM}_{\theta}$ is naturally equipped with a signed symmetric pairing $\langle\cdot, \cdot\rangle_{\theta}$. We define $\overline{\rmG}_{\theta}=\mathrm{GAut}(\overline{\rmM}_{\theta}, \langle\cdot, \cdot\rangle_{\theta})$ which is a orthogonal or symplectic similitude group. For each $\theta\in\calT_{3}$, we consider $\theta\oplus\theta^{\ast}$ and $\overline{\rmM}_{\theta\oplus\theta^{\ast}}=\Hom_{\rmP_{\rmF_{v}}}(\theta\oplus\theta^{\ast}, \overline{\rho}_{v})$ equipped with its signed symmetric pairing $\langle\cdot, \cdot\rangle_{\theta\oplus\theta^{\ast}}$. We define $\overline{\rmG}_{\theta}=\mathrm{GAut}(\overline{\rmM}_{\theta\oplus\theta^{\ast}}, \langle\cdot, \cdot\rangle_{\theta\oplus\theta^{\ast}})$. In all cases, we can lift $\overline{\rmG}_{\theta}$ to a split reductive group $\rmG_{\theta}$ over $\calO$. In the cases $(1), (2)$, $\rmM_{\theta}$ gives a representation of $\rmT_{\rmF_{v}, \theta}$ valued in $\rmG_{\theta}(A)$ and in the case $(3)$, $\rmM_{\theta\oplus\theta^{*}}$ gives a representation of $\rmT_{\rmF_{v}, \theta\oplus\theta^{*}}$ valued in $\rmG_{\theta}(A)$ where $\rmT_{\rmF_{v}, \theta\oplus\theta^{*}}$ is defined in the same way as $\rmT_{\rmF_{v}, \theta}$. Note that in each of the above cases, $\rmT_{\rmF_{v},\theta}$ and $\rmT_{\rmF_{v}, \theta\oplus\theta^{*}}$ are groups of the form $\rmT_{q^{\prime}}$ for some $q^{\prime}$ a power of $q$. 

\begin{definition}\label{min-ram}
Let $\rho_{v}: \rmG_{\rmF_{v}}\rightarrow \GSp_{4}(A)$ be a lifting of $\overline{\rho}_{v}$ for some $A\in\CNL_{\calO}$ with fixed similitude character $\psi=\epsilon^{-3}_{l}$. Let $\rmM$ be the $A[\rmG_{\rmF_{v}}]$-module associated to $\rho_{v}: \rmG_{\rmF_{v}}\rightarrow \GSp_{4}(A)$ with the decomposition
\begin{equation*}
\begin{aligned}
\rmM=&\bigoplus\limits_{[\theta]\in\calT_{1}/\rmG_{\rmF_{v}}}\mathrm{Ind}^{\rmG_{\rmF_{v}}}_{\rmG_{\rmF_{v},\theta}}(\tilde{\theta}\otimes\rmM_{\theta})\oplus \bigoplus\limits_{[\theta]\in\calT_{2}/\rmG_{\rmF_{v}}}\mathrm{Ind}^{\rmG_{\rmF_{v}}}_{\rmG_{\rmF_{v},\theta}}(\tilde{\theta}\otimes\rmM_{\theta})\oplus \\
&\bigoplus\limits_{[\theta]\in\calT_{3}/\rmG_{\rmF_{v}}}\mathrm{Ind}^{\rmG_{\rmF_{v}}}_{\rmG_{\rmF_{v},\theta\oplus\theta^{\ast}}}((\tilde{\theta}\oplus\tilde{\theta^{\ast}})\otimes\rmM_{\theta\oplus\theta^{\ast}}).\\
\end{aligned}
\end{equation*}
We say $\rho_{v}$ is minimally ramified with similitude factor $\epsilon^{-3}_{l}$ if each $\rmM_{\theta}$ and $\rmM_{\theta\oplus\theta^{\ast}}$ is minimally ramified in the sense of Definition \ref{min-tame}.
\end{definition}

Let $\calD^{\mathrm{min}}_{v}$ the local deformation problem classifying all the minimally ramified liftings of $\overline{\rho}_{v}$ with fixed similitude character $\epsilon^{-3}_{l}$. Concerning $\calD^{\mathrm{min}}_{v}$,  we have the following result.
\begin{proposition}\label{min-dim}
The local universal lifting ring $\rmR^{\square}_{v}$ is reduced, local complete intersection, flat and pure of relative dimension $10$ over $\calO$. The local deformation problem $\calD^{\mathrm{min}}_{v}$ which is represented by $\mathrm{Spf}\phantom{.}{\mathrm{R}^{\mathrm{min}}_{v}}$ gives an irreducible component of $\mathrm{Spf}\phantom{.}{\mathrm{R}^{\square}_{v}}$ which is formally smooth, pure of relative dimension $10$ over $\calO$.
\end{proposition}
\begin{proof}
The exact proof of \cite[Proposition 3.4.12]{LTXZZa} works here.
\end{proof}

\subsection{Level raising deformations}
Next we discuss the notion of \emph{level raising deformation}. Let $v$ be a place of $\rmF$ whose norm is $q_{v}$  and such that $q^{2}_{v}\not\equiv 1\mod l$.  Suppose $\overline{\rho}_{v}$ is an unramified representation with similitude character $\epsilon^{-3}_{l}$. Then it follows that every liftings of $\overline{\rho}_{v}$ can be viewed as an representation of the $q_{v}$-tame group $\rmT_{\mathrm{F}_{v}}$ by Lemma \ref{wild-rep}. Suppose that the generalized eigenvalues of $\overline{\rho}_{v}(\phi_{v})$ contains $\{\rmu q^{-1}_{v}, \rmu q^{-2}_{v}\}$ exactly once for any $\rmu\in\{\pm1\}$. Let $\rho_{v}$ be a deformation of $\overline{\rho}_{v}$ over $A$, by Lemma \ref{decomp}, there is a decomposition $A^{4}=\rmM_{0}\oplus\rmM_{1}$ stable under the action of $\rho_{v}(\phi_{v})$. Let $\rmP_{0}(\rmT)$ be the characteristic polynomial of $\rho_{v}(\phi_{v})$ on $\rmM_{0}$ with $\rmP_{0}(\rmT)\equiv (\rmT-\rmu q^{-1}_{v})(\rmT-\rmu q^{-2}_{v}) \mod \fracm_{A}$. 

\begin{definition}
Let $(\overline{\rho}_{v}, \epsilon^{-3}_{l})$ be as above. We define the following local deformation problems.
\begin{enumerate}
\item $\calD^{\mathrm{mix}}_{v}$ is the local deformation problem classifying the liftings $\rho_{v}$ of $\overline{\rho}_{v}$ to $A$ of $\CNL_{\calO}$ such that in the decomposition $A^{4}=\rmM_{0}\oplus\rmM_{1}$ of $\rho_{v}$, $\rmM_{0}$ is stable under $\rho_{v}(\rmI_{\rmF_{v}})$ and $\rmM_{1}$ is unramified. 
\item $\calD^{\mathrm{un}}_{v}$ is the local deformation problem such that the action of $\rho_{v}(\mathrm{I}_{\rmF_{v}})$ on $\rmM_{0}$ is trivial;
\item $\calD^{\mathrm{ram}}_{v}$ is the local deformation problem such that the action of $\rho_{v}(\phi_{v})$ has characteristic polynomial given by $\rmP_{0}(\rmT)= (\rmT-\rmu q^{-1}_{v})(\rmT- \rmu q^{-2}_{v})$ on $\rmM_{0}$.
\end{enumerate}
\end{definition}
It is clear that the two local deformation problems  $\calD^{\mathrm{un}}_{v}$ and $\calD^{\mathrm{ram}}_{v}$ are contained in $\calD^{\mix}_{v}$. In fact, we have the following proposition which is important in the applications.
\begin{proposition}\label{mix-dim}
Suppose that $l\nmid q^{2}_{v}-1$ and that the generalized eigenvalues of $\overline{\rho}_{v}(\phi_{v})$ in $\overline{\FF}_{l}$ contain the pair $(\rmu q^{-1}_{v}, \rmu q^{-2}_{v})$ exactly once. Then the formal scheme $\Spf\phantom{.}\mathrm{R}^{\mathrm{mix}}_{v}$representing $\calD^{\mathrm{mix}}_{v}$ is formally smooth over $\calO[[x_{0}, x_{1}]]/(x_{0}x_{1})$ of pure relative dimension $10$ whose irreducible component corresponding to $x_{0}=0$ is $\Spf\phantom{.} \rmR^{\mathrm{un}}_{v}$ representing the local deformation problem $\calD^{\mathrm{un}}_{v}$ and whose irreducible component corresponding to $x_{1}=0$ is $\Spf\phantom{.}\rmR^{\mathrm{ram}}_{v}$ representing the local deformation problem $\calD^{\mathrm{ram}}_{v}$.
\end{proposition}

\begin{proof}
This proposition can be proved in the same way as in \cite[Proposition 3.5.2]{LTXZZa}. Let $k^{4}=\overline{\rmM}_{0}\oplus \overline{\rmM}_{1}$ be the decomposition of $\overline{\rho}_{v}$ such that $\overline{\rho}_{v}(\phi_{v})$ has eigenvalues $(\rmu q^{-1}, \rmu q^{-2})$ on $\overline{\rmM}_{0}$. After choosing a suitable basis, we can assume that $\overline{\rmM}_{0}$ is spanned by the first two vectors and $\overline{\rmM}_{1}$ is spanned by the last two vectors. Therefore we obtain two unramified morphisms $\overline{\rho}_{0, v}: \rmT_{v}\rightarrow \GL_{2}(k)$ and $\overline{\rho}_{1, v}: \rmT_{v}\rightarrow \GL_{2}(k)$. Let $\calD_{0}$ be the local deformation problem classifying liftings of $\overline{\rho}_{0, v}$ and $\calD_{1}$ be the local deformation problem classifying unramified liftings of $\overline{\rho}_{1, v}$. Let $A\in \CNL_{\calO}$. We say a lifting $\rho_{v}: \rmT_{v}\rightarrow \GSp_{4}(A)$ a standard lifting of $\overline{\rho}_{v}$ to $A$ if
\begin{equation*}
\begin{aligned}
&\rho_{v}(t)=\begin{pmatrix} \rmA_{0} & 0\\ 0& \rmI_{2}\\ \end{pmatrix};\\
&\rho_{v}(\phi_{v})=\begin{pmatrix} \mathrm{B}_{0} & 0\\ 0& \rmB_{1}\\ \end{pmatrix}\\
\end{aligned}
\end{equation*}
with $\rmA_{0}, \rmB_{0}, \rmB_{1}\in \GL_{2}(A)$. Let $\calD^{\mathrm{mix}}_{0,1}$ be the subfunctor classifying standard liftings. Note that $\calD^{\mathrm{mix}}_{0,1}=\calD_{0}\times_{\Spf\phantom{.}\calO}\calD_{1}$. 
Moreover we have an isomorphism 
\begin{equation*}
\calD^{\mathrm{mix}}_{0,1}\times_{\Spf\phantom{.}\calO} (\widehat{\PGSp_{4}}/\widehat{\PGL_{2}}\times_{\Spf\phantom{.}\calO} \widehat{\PGL_{2}})\cong \calD^{\mathrm{mix}}.
\end{equation*}
By \cite[Proposition 5.5(2)]{Sho-GL2}, we have $\calD_{1}$ is formally smooth of relative dimension $3$ over $\Spf\phantom{.}\calO$ and $\calD_{0}$ is isomorphic to $\Spf\phantom{.}\calO[[\rmX_{1},\cdots, \rmX_{4}]]/(\rmX_{1}\rmX_{2})$ where the irreducible component for $\rmX_{0}=0$ classifies the unramified liftings and the irreducible component for $\rmX_{1}=0$ classifies the ramified liftings. The result follows

Alternatively one can reduce this result to the proof of \cite[Proposition 3.5.2]{LTXZZa}. Consider the group scheme $\mathcal{G}_{4}$ be the group scheme over $\ZZ$ which is the semi-direct product of $\GL_{4}\times \GL_{1}$ by the group $\{1, j\}$ which acts on $\GL_{4}\times \GL_{1}$ by
\begin{equation*}
j(g, \mu)j^{-1}=(\mu \prescript{t}{}{g^{-1}}, \mu).
\end{equation*}
There is a homomorphism $\nu: \mathcal{G}_{4}\rightarrow \GL_{1}$ sending $(g, \mu)$ to $\mu$ and $j$ to $-1$. Let $A$ be a complete local Noetherian ring in $\CNL_{\calO}$ and $\rho_{v}: G_{\rmF_{v}}\rightarrow \GSp_{4}(A)$. We choose $\rmM_{v}$ a quadratic extension of $\rmF_{v}$ disjoint from $\overline{\rmF_{v}}^{\mathrm{ker}(\overline{\rho})}$. Then there is a continuous homomorphism $r_{v}: G_{\rmF_{v}}\rightarrow \mathcal{G}_{4}(R)$ determined by 
\begin{equation*}
\begin{aligned}
&r_{v}(g)=(\rho_{v}(g), \nu\circ \rho_{v}(g)) \text{ for $g\in G_{\rmM_{v}}$};\\
&r_{v}(g)=(\rho_{v}(g)J, -\nu\circ \rho_{v}(g))j \text{ for $g\not\in G_{\rmM_{v}}$}.\\
\end{aligned}
\end{equation*}
In particular, one can associate $\overline{\rho}_{v}$ the representation $\overline{r}_{v}$ valued in $\calG_{4}(k)$. One sees immediately that our deformation problem $\calD^{\mathrm{mix}}_{v}$ corresponds to the deformation problem $\calD^{\mathrm{mix}}_{v}$ for $\overline{r}_{v}$ as in \cite[Definition 3.5.1]{LTXZZa}. Then one can proceed the same way as in \cite[Proposition 3.5.2]{LTXZZa} with obvious modifications.
\end{proof}

\subsection{Fontaine-Laffaille deformations} We recall the theory of Fontaine-Laffaille modules and Fontaine-Laffaille deformations. Let $v$ be a place of $\rmF$ above $l$ which we assume is unramfied in $\rmF$. Recall $E_{\lambda}$ be the coefficient field and $\calO\subset E_{\lambda}$ be the valuation ring. We will temporarily make the following assumption on $E_{\lambda}$.
\begin{assumption}\label{strong-field}
The field $E_{\lambda}$ contains the image of all embeddings of $\rmF_{v}$ in $\overline{\QQ}_{l}$.
\end{assumption}
Let $\sigma\in \Gal(\rmF_{v}/\QQ_{l})$ be the absolute Frobenius. Let $\Sigma_{v}=\Hom_{\ZZ_{l}}(\calO_{\rmF_{v}}, \calO)$. Let $\calM\calF_{\calO, v}$ be the category of $\calO_{\rmF_{v}}\otimes \calO$-modules $\rmM$ of finite length with
\begin{itemize}
\item  a decreasing filtration $\{\mathrm{Fil}^{i}\mathrm{M}\}_{i\in\ZZ}$ by $\calO_{\rmF_{v}}\otimes\calO$-submodules that are $\calO_{\rmF_{v}}$-direct summands, satisfying $\mathrm{Fil}^{0}\mathrm{M}=\mathrm{M}$ and $\mathrm{Fil}^{l-1}\mathrm{M}=0$.
\item a $\sigma\otimes 1$-linear map $\Phi^{i}: \mathrm{Fil}^{i}\mathrm{M}\rightarrow \mathrm{M}$ for $i\in \ZZ$ with $\Phi^{i}\vert_{{\mathrm{Fil}^{i+1}\mathrm{M}}}=l\Phi^{i+1}$ and $\sum_{i\in\ZZ}\Phi^{i}{\mathrm{Fil}^{i}\mathrm{M}}=\mathrm{M}$.
\end{itemize}
We define the following subcategories of  $\calM\calF_{\calO, v}$:
\begin{itemize}
\item $\calM\calF_{k, v}$ is the full subcategory of $\calM\calF_{\calO, v}$ annihilated by $\lambda$.
\item Let $0\leq b\leq l-2$, $\calM\calF^{[0, b]}_{\calO, v}$ is the full subcategory of $\calM\calF_{\calO, v}$  consisting those $\rmM$ with $\mathrm{Fil}^{b+1}\rmM=0$.
\end{itemize}
Note that there is a decomposition $\rmM=\bigoplus\limits_{\tau\in\Sigma_{v}}\rmM_{\tau}$ corresponding to the decomposition $\calO_{\rmF_{v}}\otimes \calO=\bigoplus\limits_{\tau\in\Sigma_{v}}\calO$ and $\Phi^{i}$ induces an $\calO$-linear map $\Phi^{i}_{\tau}: \mathrm{Fil}^{i}\rmM_{\tau}\rightarrow \rmM_{\tau\circ\sigma^{-1}}$. We put
\begin{equation*}
\begin{aligned}
\mathrm{gr}^{i}\rmM_{\tau}=\mathrm{Fil}^{i}\rmM_{\tau}/\mathrm{Fil}^{i+1}\rmM_{\tau}.
\end{aligned}
\end{equation*}
We define the set of \emph{$\tau$-Fontaine-Laffaille weights} of $\rmM$ to be
\begin{equation*}
\mathrm{HT}_{\tau}(\rmM):=\{i\in\ZZ: \mathrm{gr}^{i}\rmM_{\tau}\neq 0\}.
\end{equation*}
We say $\rmM$ has \emph{regular Fontaine-Laffaille weights} if $\mathrm{gr}^{i}\mathrm{M}_{\tau}$ is generated over $\calO$ by most one element for every $\tau\in\Sigma_{v}$ and every $i\in\ZZ$. 

Let $\calO[\rmG_{\rmF_{v}}]^{\mathrm{fl}}$ be the category of $\calO$-modules of finite length with a continuous action of $\rmG_{\rmF_{v}}$. There is an exact fully faithful, covariant $\calO$-linear functor \cite[2.4.1]{CHT}
\begin{equation*}
\mathbf{G}_{v}: \calM\calF_{\calO, v}\rightarrow \calO[\rmG_{\rmF_{v}}]^{\mathrm{fl}}.
\end{equation*}
The length of $\rmM$ in $\calM\calF_{\calO, v}$ as an $\calO$-module equals to $[\rmF_{v}:\QQ_{l}]$ times the length of $\mathbf{G}_{v}(\rmM)$ as an $\calO$-module.

\begin{definition}
For an integer $b$ with $0\leq b\leq l-2$ and let $A\in\ART_{\calO}$. Let $A\{b\}$ be the object in $\calM\calF_{\calO, v}$ defined by the $\calO_{\rmF_{v}}\otimes\calO$-module of rank $1$ given by $(\calO_{\rmF_{v}}\otimes\calO)e_{b}$ with filtration given by
\begin{equation*}
\begin{aligned}
&\mathrm{Fil}^{i}A\{b\}=(\calO_{\rmF_{v}}\otimes A)e_{b}\text{ for $i\leq b$};\\
&\mathrm{Fil}^{i}A\{b\}=0 \text{ for $i> b$}. \\
\end{aligned}
\end{equation*}
with $\Phi^{b}(e_{b})=e_{b}$. Then we have $\mathbf{G}_{v}(A\{b\})\cong A(-b)\vert_{\rmG_{\rmF_{v}}}$ as $\calO[\rmG_{\rmF_{v}}]$-modules.
\end{definition}
Let $A\in\ART_{\calO}$ and $0\leq b\leq l-2$. We define $\calM\calF^{[0, b]}_{\calO, v}\vert_{A}$ be the full subcategory of  $\calM\calF^{[0, b]}_{\calO, v}$ consisting of those $\rmM$ that are finite free $\calO_{\rmF_{v}}\otimes A$-modules such that $\Fil^{i}\rmM$ are direct summand of $\rmM$ for every $i$. The functor $\mathbf{G}_{v}$ restrict to a functor $\mathbf{G}_{v}: \calM\calF^{[0, b]}_{\calO, v}\vert_{A}\rightarrow A[\rmG_{\rmF_{v}}]^{\mathrm{f.r}}$ where $A[\rmG_{\rmF_{v}}]^{\mathrm{f.r}}$ is the category of finite free $A$-modules equipped with a continuous $\rmG_{\rmF_{v}}$-action. Let $\rmM\rightarrow \rmM^{\vee}\{b\}$ be the functor  $\calM\calF^{[0, b]}_{\calO, v}\vert_{A}\rightarrow \calM\calF^{[0, b]}_{\calO, v}\vert_{A}$ defined by the following recipe:
\begin{itemize}
\item the underlying $\calO_{\rmF_{v}}\otimes A$-module of $\rmM^{\vee}\{b\}$ is given by $\Hom_{\calO_{\rmF_{v}}\otimes A}(\rmM, \calO_{\rmF_{v}}\otimes A)$;
\item $\Fil^{i}\rmM^{\vee}\{b\}=\Hom_{\calO_{\rmF_{v}}\otimes A}(\rmM/\Fil^{b+1-i}\rmM, \calO_{\rmF_{v}}\otimes A)$;
\item for every $f\in \Fil^{i}\rmM^{\vee}\{b\}$ and $m\in\Fil^{j}\rmM$, we have
\begin{equation*}
\begin{aligned}
&\Phi^{i}(f)(\Phi^{j}(m))=l^{b-i-j}f(m) \text{ for $i+j\leq b$};\\
&\Phi^{i}(f)(\Phi^{j}(m))=0 \text{ for $i+j> b$}.\\
\end{aligned}
\end{equation*}
\end{itemize}
In fact, we have $\mathbf{G}_{v}(\rmM^{\vee}\{b\})=\mathbf{G}_{v}(\rmM)^{\vee}(-b)$. 

Suppose that $\mathbf{G}_{v}(\rmM)=\mathbf{G}_{v}(\rmM)^{\vee}(-b)$. By the above discussion, we have an isomorphism $\rmM^{\vee}\{b\}\cong\rmM$ in $\calM\calF^{[0, b]}_{\calO, v}\vert_{A}$. This isomorphism induces a natural pairing 
\begin{equation*}
\langle\cdot, \cdot\rangle: \rmM\times\rmM\rightarrow \calO_{\rmF_{v}}\otimes A.
\end{equation*}
We say $\rmM$ is \emph{symplectic} if $\langle m, n \rangle=-\langle n, m \rangle$ for all $m, n\in \rmM$. 

\begin{definition}\label{FL}
Let $A$ be an object in $\CNL_{\calO}$ and $\rho_{v}:G_{\rmF_{v}}\rightarrow \GSp_{4}(A)$ be a continuous representation. 
\begin{enumerate}
\item Let $a, b$ be integers satisfying $0\leq b-a\leq l-2$. For $E_{\lambda}$ satisfying Assumption \ref{strong-field}, we say $\rho_{v}$ is   regular Fontaine-Laffaille crystalline with Fontaine-Laffaille weights in $[a, b]$ if, for every quotient $\overline{A}$ of $A$ in $\ART_{\calO}$, $\rho_{v}(a)\otimes_{A}\overline{A}$ lies in the essential image of the functor $\mathbf{G}_{v}: \calM\calF^{[0, b-a]}_{\calO, v}\vert_{\overline{A}}\rightarrow \overline{A}[\rmG_{\rmF_{v}}]^{\mathrm{f.r}}$ and that $\mathbf{G}^{-1}_{v}(\rho_{v}(a)\otimes_{A}\overline{A})$ has regular Fontaine-Laffaille weights and is symplectic.
\item Now for $E_{\lambda}$ not necessarily satisfy Assumption \ref{strong-field}, we say that $\rho_{v}$ is regular Fontaine-Laffaille crystalline if there exists a finite unramified extension $E^{\dagger}_{\lambda}$ of $E_{\lambda}$ in $\overline{\QQ}_{l}$ with ring of integers $\calO^{\dagger}$ satisfying Assumption \ref{strong-field} such that  $\rho_{v}\otimes_{\calO}\calO^{\dagger}$ is crystalline with regular Fontaine-Laffaille weights in $[a, b]$ as in $(1)$.
\end{enumerate}
\end{definition}

Now we are in the position to define the local deformation problem at places over $l$. 

\begin{definition}
Let $\calD^{\mathrm{FL}}_{v}$ be the local deformation problem that classifies liftings $\rho_{v}$ of $\overline{\rho}_{v}$ that are regular Fontaine-Laffaille crystalline and similitude character $\epsilon^{-3}_{l}$. 
\end{definition}

\begin{proposition}\label{FL-dim}
Suppose $\overline{\rho}_{v}$ is regular Fontaine-Laffaille crystalline with similitude factor $\epsilon^{-3}_{l}$. The local deformation problem $\calD^{\mathrm{FL}}_{v}$ is formally smooth of relative dimension $10+4[\rmF_{v}: \QQ_{l}]$ over $\Spf\phantom{.}\calO$.
\end{proposition}
\begin{proof}
This follows from the \cite[Theorem 5.2]{Boo19b} and the fact that $\dim_{k}\mathrm{ad}^{0}\overline{\rho}_{v}=10$ and $\dim_{k}\GSp_{4}(k)/\rmB(k)=4$. See also \cite[Proposition 7.2.1]{GG-companion}
\end{proof}

\section{Cohomology of certain quaternionic unitary Shimura variety}
Let $\rmB$ a quaternion algebra over $\rmF$. Let $\mathbf{G}(\rmB)=\mathrm{Res}_{\rmF/\QQ}\mathrm{GU}_{2}(\mathrm{B})$ be the quaternionic unitary similitude group. This is an inner form of the group $\mathbf{G}=\mathrm{Res}_{\rmF/\QQ}\mathrm{GSp}_{4}$. We will only be interested in the following cases. Let $\Sigma_{\infty}$ be the set of infinity places of $\rmF$. 
\begin{itemize}
\item The quaternion algebra $\rmB$ is totally definite that is $\rmB$ is ramified at all infinite places $\Sigma_{\infty}$. We will refer to this case as the definite case.
\item The quaternion algebra $\rmB$ is ramified at all infinite places $\Sigma_{\infty}$ except for one infinite place $v^{\mathrm{split}}_{\infty}$.  We will refer to this case as the indefinite case.
\end{itemize}
We will write $\Sigma_{\rmB}$ as the set of non-archemedean places of $\rmF$ such that $\rmB$ is ramified at. Let $\xi=(\xi_{v})_{v\mid \infty}\in (\ZZ^{3}_{\geq 0})^{\Sigma_{\infty}}$ be an element giving a highest weight of  $\prod_{v\mid\infty}\GSp_{4}(\CC)$. We can regard it as a highest weight of 
\begin{equation*}
\mathbf{G}(\rmB)(\CC)=\mathrm{Res}_{\rmF/\QQ}\mathrm{GU}_{2}(\mathrm{B})(\CC)=\prod_{v\mid \infty} \mathrm{GU}_{2}(\mathrm{B})(\CC).
\end{equation*}
since $\GU_{2}(\rmB)(\CC)\cong\GSp_{4}(\CC)$. We can write each $\xi_{v}$ as $(\xi_{1, v}, \xi_{2, v}; \tilde{\xi}_{v})=(k_{v}-1, l_{v}-2, w_{v})$. We will assume that $\xi_{1, v}+\xi_{2, v}\leq l-2$ or equivalently $k_{v}+l_{v}\leq l+1$. We will also assume that $w_{v}=0$ for each $v\mid \infty$. We will assume that $l$ is unramified in $\rmF$. Fix an isomorphism $\iota_{l}: \CC\cong \overline{\QQ}_{l}$. We will choose the coefficient field $E_{\lambda}$ such that the complex algebraic representation $\xi$ of $\mathrm{Res}_{\rmF/\QQ}\mathrm{GU}_{2}(\mathrm{B})$ is defined over $\iota^{-1}_{l}E_{\lambda}$.

Let $(\overline{\rho}, \psi)$ be a pair with $\overline{\rho}:G_{\rmF}\rightarrow \GSp_{4}(k)$ and $\psi=\epsilon^{-3}_{l}$. Let $\Sigma_{\mathrm{min}}$ and $\Sigma_{\mathrm{lr}}$ be two sets of non-archimedean places of $\rmF$ such that
\begin{itemize}
\item $\Sigma_{\mathrm{min}}$, $\Sigma_{\mathrm{lr}}$ and $\Sigma_{l}$ are mutually disjoint;
\item $\Sigma_{\mathrm{lr}}$ contains $\Sigma_{\mathrm{B}}$;
\item for every $v\in \Sigma_{\mathrm{lr}}$ such that $l\nmid (q^{2}_{v}-1)$.
\end{itemize}

\begin{definition}\label{rigid}
We say $\overline{\rho}$ is rigid for $(\Sigma_{\mathrm{min}}, \Sigma_{\mathrm{lr}})$ if the following are satisfied.
\begin{enumerate}
\item For $v\in \Sigma_{\mathrm{min}}$, every lifting of $\overline{\rho}_{v}$ is minimally ramified as in Definition \ref{min-ram}.
\item For $v\in \Sigma_{\mathrm{lr}}$, the generalized eigenvaules of $\overline{\rho}_{v}$ contain the pair $\{\rmu q_{v}^{-1}, \rmu q_{v}^{-2}\}$ exactly once for $u\in{\pm 1}$.
\item For $v\in\Sigma_{l}$, $\overline{\rho}_{v}$ is regular Fontaine-Laffaille crystalline as in Definition \ref{FL}.
\item For $v\not\in\Sigma_{\mathrm{min}}\cup\Sigma_{\mathrm{lr}}\cup\Sigma_{l}$, the representation $\overline{\rho}_{v}$ is unramified.
\end{enumerate}
\end{definition}

Suppose that $\overline{\rho}$ is rigid for $(\Sigma_{\mathrm{min}}, \Sigma_{\mathrm{lr}})$. We consider a global deformation problem 
\begin{equation*}
\calS=(\overline{\rho}, \epsilon^{-3}_{l}, \Sigma_{\mathrm{min}}\cup \Sigma_{\mathrm{lr}}\cup \Sigma_{l}, \{\calD_{v}\}_{v\in\Sigma_{\mathrm{min}}\cup \Sigma_{\mathrm{lr}} \cup \Sigma_{l}})
\end{equation*}
where 
\begin{itemize}
\item for $v\in\Sigma_{\mathrm{min}}$, $\calD_{v}$ is the local deformation problem classifying all liftings of $\overline{\rho}_{v}$;
\item for $v\in\Sigma_{\mathrm{lr}}$, $\calD_{v}$ is the local deformation problem $\calD^{\mathrm{lr}}_{v}$;
\item for $v\in \Sigma_{l}$, $\calD_{v}$ is the local deformation problem $\calD^{\mathrm{FL}}_{v}$.
\end{itemize}
Then we have the global universal deformation ring $\rmR^{\univ}_{\calS}$ classifying all deformations of $\overline{\rho}$ of type $\calS$. 

We choose $\rmB$ such its ramified place $\Sigma_{\rmB}\subset\Sigma_{\lr}$. We also choose a neat open compact subgroup $\rmK$ of $\mathbf{G}(\rmB)(\mathbb{A}^{\infty}_{\rmF})$ of the form
\begin{equation*}
\rmK=\prod\limits_{v\in\Sigma_{\min}\cup\Sigma_{\lr}}\rmK_{v}\times\prod\limits_{v\not\in\Sigma_{\min}\cup\Sigma_{\lr}}\GSp_{4}(\calO_{\rmF_{v}})
\end{equation*}
where $\rmK_{v}$ is the paramodular subgroup of $\mathrm{GU}_{2}(\rmB_{v})$ for those $v$ ramified in $\rmB$ and $\rmK_{v}$ is the paramodular subgroup of $\mathrm{GSp}_{4}(\rmF_{v})$ for the other $v$ in $\Sigma_{\lr}$, and $\rmK_{v}$ is the {pro-$v$ Iwahori subgroup} of $\GSp_{4}(\rmF_{v})$ for $v\in\Sigma_{\mathrm{min}}$. We have the system of Shimura varieties 
\begin{equation*}
\{\mathrm{Sh}(\rmB, \rmK^{\prime})\}_{\rmK^{\prime}\subset \rmK}
\end{equation*}
associated to $\mathbf{G}(\rmB)$ which are are quasi-projective smooth complex schemes of dimension $d(\rmB)$. When $\rmB$ is definite, then $d(\rmB)=0$ and when $\rmB$ is indefinite, then $d(\rmB)=3$. There is a homorphism
\begin{equation*}
\prod_{v\in\Sigma_{l}}\GSp_{4}(\calO_{\rmF_{v}})\rightarrow \GL(L_{\xi})
\end{equation*}
for a finite free $\calO$-module $L_{\xi}$ associated to $\xi$. This defines an $\calO$-local system $\calL_{\xi}$ on the system of Shimura varieties $\{\mathrm{Sh}(\rmB, \rmK^{\prime})\}_{\rmK^{\prime}\subset \rmK}$. 
Consider the spherical Hecke algebra at $v$
\begin{equation*}
\mathbb{T}_{v}:=\calO[\GSp_{4}(\calO_{\rmF_{v}})\backslash\GSp_{4}(\rmF_{v})/\GSp_{4}(\calO_{\rmF_{v}})].
\end{equation*}
Let $\{\rmT_{v, i}\}_{i=0,1,2}$ be elements given by $[\GSp_{4}(\calO_{\rmF_{v}})\beta_{v, i}\GSp_{4}(\calO_{\rmF_{v}})]$ where
\begin{equation*}
\begin{aligned}
&\beta_{v, 0}=\mathrm{diag}(\varpi_{v},\varpi_{v}, \varpi_{v}, \varpi_{v});\\
&\beta_{v, 1}=\mathrm{diag}(\varpi^{2}_{v},\varpi_{v}, \varpi_{v}, 1);\\
&\beta_{v, 2}=\mathrm{diag}(\varpi_{v},\varpi_{v}, 1, 1).\\
\end{aligned}
\end{equation*}
Then it is known that $\TT_{v}$ is generated by $\rmT_{v, i}$. Let $\Sigma$ be a finite set of places of $\rmF$ containing $\Sigma_{\min}\cup\Sigma_{\lr}$. We define the abstract Hecke algebra to be $\mathbb{T}^{\Sigma}:=\bigotimes\limits_{v\not\in\Sigma}\mathbb{T}_{v}$. Let $\fracm$ be the maximal ideal of $\mathbb{T}^{\Sigma}$ corresponding to $\overline{\rho}$. Thus $\fracm$ contains $\lambda$ and is such that the characteristic polynomial $\mathrm{det}(\rmX-\overline{\rho}(\phi^{-1}_{v}))$ is congruent to the polynomial $\mathrm{Q}_{v}(\rmX)$ modulo $\fracm$ defined by
\begin{equation*}
\rmX^{4}-\rmT_{v,2}\rmX^{3}+(q_{v}\rmT_{v,1}+(q^{3}_{v}+q_{v})\rmT_{v, 0})\rmX^{2}-q^{3}_{v}\rmT_{v, 0}\rmT_{v, 2}\rmX+q^{6}_{v}\rmT^{2}_{v, 0}
\end{equation*}
where $q_{v}=\Vert v\Vert$ is the norm of $v$.

\begin{theorem}\label{main}
Under the above setup, suppose the following assumptions hold.
\begin{itemize}
\item[(D0)] $l\geq 3$ is a prime unramified in $\rmF$. 
\item[(D1)] The weight $(k_{v}, l_{v})_{v\mid\infty}$ of $\pi$ satisfies $l_{v}+k_{v}\leq l+1$;
\item[(D2)] The image $\overline{\rho}_{\pi,\lambda}(\rmG_{\rmF})$ contains $\GSp_{4}(\FF_{l})$;
\item[(D3)] $\overline{\rho}_{\pi, \lambda}$ is rigid for $(\Sigma_{\mathrm{min}}, \Sigma_{\mathrm{lr}})$;
\item[(D4)] Suppose that $\rmB$ is indefinite. For every finite set $\Sigma^{\prime}$ of nonarchimedean places of $\rmF$ containing $\Sigma$, and every open compact subgroup $\rmK^{\prime}\subset \rmK$ satisfying $\rmK^{\prime}_{v}=\rmK_{v}$ for $v\not\in \Sigma^{\prime}$, we have 
\begin{equation*}
\rmH^{d}(\mathrm{Sh}(\rmB, \rmK^{\prime}), \calL_{\xi}\otimes_{\calO}k)_{\fracm^{\prime}}=0
\end{equation*}
for $d\neq 3$ where $\fracm^{\prime}=\fracm\cap \TT^{\Sigma^{\prime}}$.
\end{itemize}
Let $\mathbf{T}$ be the image of $\TT^{\Sigma}$ in $\mathrm{End}_{\calO}(\rmH^{d(\rmB)}(\mathrm{Sh}(\rmB, \rmK), \calL_{\xi}))$ and suppose that $\mathbf{T}_{\fracm}$ is non-zero. Then the following results hold.
\begin{enumerate}
\item There is an isomorphism of complete intersection rings: $\rmR^{\univ}_{\calS}\cong\mathbf{T}_{\fracm}$.
\item The localized cohomology $\rmH^{d(\rmB)}(\mathrm{Sh}(\rmB, \rmK), \calL_{\xi})_{\fracm}$ is finite free module over $\mathbf{T}_{\fracm}$ and hence $\rmH^{d(\rmB)}(\mathrm{Sh}(\rmB, \rmK), \calL_{\xi})_{\fracm}$ is also finite free over $\rmR^{\univ}_{\calS}$.
\end{enumerate}
\end{theorem}
\begin{remark}
We remark that the big image assumption $(\mathrm{D1})$ is for simplicity. In fact it suffices to assume that the representation $\overline{\rho}$ is vast and tidy in the sense of \cite[Definition 7.5.6, 7.5.11]{BCGP} which is weaker assumption.
\end{remark}

\subsection{Taylor-Wiles System} We fix a global deformation problem 
\begin{equation*}
\calS=(\overline{\rho}, \epsilon^{-3}_{l}, S=\Sigma_{\mathrm{min}}\cup \Sigma_{\mathrm{lr}}\cup \Sigma_{\mathrm{l}}, \{\calD_{v}\}_{v\in\Sigma_{\mathrm{min}}\cup \Sigma_{\mathrm{lr}} \cup \Sigma_{l}}).
\end{equation*} 
Then we define a \emph{Taylor-Wiles datum} $(\rmQ, (\overline{\alpha}_{v}, \overline{\beta}_{v}, \overline{\gamma_{v}}, \overline{\delta_{v}}))$  for $\calS$ consisting of the following:
\begin{itemize}
\item $\rmQ$ is finite set of primes of $\rmF$ away from $S=\Sigma_{\min}\cup\Sigma_{\lr}\cup\Sigma_{l}$, such that $q_{v}\equiv 1\mod l$ for all $v\in \rmQ$;
\item For each $v\in\rmQ$, the Frobenius eigenvalues $(\overline{\alpha}_{v}, \overline{\beta}_{v}, \overline{\gamma}_{v}=\overline{\beta}^{-1}_{v}, \overline{\delta}_{v}=\overline{\alpha}^{-1}_{v})$ of $\overline{\rho}(\phi^{-1}_{v})$ are pairwise distinct and $k$-rational.
\end{itemize}
Recall that $\Iw(v)$ is the Iwahori subgroup of $\GSp_{4}(\calO_{\rmF_{v}})$ and $\Iw_{1}(v)$ is pro-$v$ Iwahori subgroup. Let $\rmT\subset  \GSp_{4}$ be the diagonal torus of $\GSp_{4}$ and $\rmW$ be the absolute Weyl group of $\GSp_{4}$. Then we have $\Iw(v)/\Iw_{1}(v)\cong \rmT(k_{v})$ where $k_{v}$ is the residue field of $v$.
\begin{itemize}
\item Let $\Delta_{v}=k^{\times}_{v}(l)^{2}$ where $k^{\times}_{v}(l)$ is the maximal $l$-power quotient of $k^{\times}_{v}$;
\item Let $\Delta_{\rmQ}=\prod_{v}\Delta_{v}$ and $\fraca_{\rmQ}$ be the augmentation ideal of $\ZZ[\Delta_{\rmQ}]$;
\item Let $\rmK_{v, 0}=\Iw(v)$ and $\rmK_{v, 1}$ be the kernel of the composite map
\begin{equation*}
\rmK_{v, 0}\rightarrow \rmT(k_{v})\cong (k^{\times}_{v})^{2}\rightarrow k^{\times}_{v}(l)^{2}
\end{equation*}
and the kernel of this map is denoted by $\rmK_{v, 1}$. Thus we have $\rmK_{v, 0}/\rmK_{v, 1}\cong \Delta_{v}$;
\item Let $\rmK_{i}(\rmQ)=\prod\limits_{v\not\in\rmQ}\rmK_{v}\times\prod\limits_{v\in\rmQ}\rmK_{v, i}$ for $i=0,1$. Then we have
\begin{equation*}
\rmK_{0}(\rmQ)/\rmK_{1}(\rmQ)\cong \Delta_{\rmQ}.
\end{equation*}
Notice that when $\rmQ=\emptyset$, then $\rmK_{0}(\rmQ)=\rmK_{1}(\rmQ)=\rmK$;
\item Let $\fracm_{\rmQ}$ be the maximal ideal of $\TT^{\Sigma\cup\rmQ}$ defined by $\fracm=\TT^{\Sigma\cup\rmQ}\cap \fracm$;
\item Finally, we let
\begin{equation*}
\rmH_{i}(\rmQ):=\Hom_{\calO}(\rmH^{d(\rmB)}(\mathrm{Sh}(\rmB, \rmK_{i}(\rmQ)), \calL_{\xi}), \calO)
\end{equation*}
for $i\in\{0,1\}$. It follows then $\rmH_{1}(\rmQ)$ is a module over $\calO[\Delta_{\rmQ}]$
\end{itemize}
\begin{lemma}\label{H-tw}
Under the assumptions of Theorem \ref{main}. The module $\rmH_{1}(\rmQ)_{\fracm_{\rmQ}}$ is finite free over $\calO[\Delta_{\rmQ}]$. Moreover $\rmH_{1}(\rmQ)_{\fracm_{\rmQ}}/\fraca_{\rmQ}=\rmH_{0}(\rmQ)_{\fracm_{\rmQ}}$.
\end{lemma}
\begin{proof}
This follows from the same proof of \cite[Lemma 3.6.5]{LTXZZa}.
\end{proof}
Given a Taylor-Wiles datum Taylor-Wiles datum $(\rmQ, (\overline{\alpha}_{v}, \overline{\beta}_{v}, \overline{\gamma_{v}}, \overline{\delta_{v}}))$, we define an augmented global deformation problem
\begin{equation*}
\calS_{\rmQ}=(\overline{\rho}, \epsilon^{-3}, \rmS\cup\rmQ,\{\calD_{v}\}_{v\in\rmS}\cup\{\calD^{\square}_{v}\}_{v\in\rmQ}).
\end{equation*}
For $v\in\rmQ$, let $\overline{\chi}_{\ast}$ for $\ast=\overline{\alpha}_{v}, \overline{\beta}_{v}, \overline{\gamma}_{v}=\overline{\beta}^{-1}_{v}, \overline{\delta}_{v}=\overline{\alpha}^{-1}_{v}$ be the unramified character sending $\phi^{-1}_{v}$ to $\ast=\overline{\alpha}_{v}, \overline{\beta}_{v}, \overline{\gamma}_{v}=\overline{\beta}^{-1}_{v}, \overline{\delta}_{v}=\overline{\alpha}^{-1}_{v}$. We have the following Lemma concerning the local deformation problem $\calD^{\square}_{v}$. 
\begin{lemma}\label{tw-lift}
Let $\rho_{v}:\rmG_{\rmF_{v}}\rightarrow \GSp_{4}(A)$ be any lift of $\overline{\rho}_{v}$. There are unique continuous characters $\chi_{\ast}: \rmG_{\rmF_{v}}\rightarrow A^{\times}$ for $\ast=\overline{\alpha}_{v}, \overline{\beta}_{v}, \overline{\gamma}_{v}, \overline{\delta}_{v}$ lifting the characters such that $\rho_{v}$ is conjugate to the lift of the form $\chi_{\overline{\alpha}_{v}}\oplus\chi_{\overline{\beta}_{v}}\oplus \epsilon^{-3}_{l}\chi_{\overline{\beta}^{-1}_{v}}\oplus \epsilon^{-3}_{l}\chi_{\overline{\alpha}^{-1}_{v}}$. It follows from this that $\Spf\phantom{.}\rmR^{\square}_{v}$ is formally smooth of relative dimension $10$ over $\calO[\Delta_{v}]$. 
\end{lemma}
\begin{proof}
The first part of the lemma follows from \cite[Lemma 5.1.1]{GT}. Let $\rho^{\univ}_{v}:\rmG_{\rmF_{v}}\rightarrow\GSp_{4}(\rmR^{\square}_{v})$ be the universal lift. Then $\rho^{\univ}_{v}$ is conjugate to a lift of the form $\chi_{\overline{\alpha}_{v}}\oplus\chi_{\overline{\beta}_{v}}\oplus \epsilon^{-3}_{l}\chi_{\overline{\beta}^{-1}_{v}}\oplus \epsilon^{-3}_{l}\chi_{\overline{\alpha}^{-1}_{v}}$. Note that the characters $\chi_{\overline{\alpha}_{v}}\circ\mathrm{Art}_{\rmF_{v}}\vert_{\calO^{\times}_{\rmF_{v}}} :\calO^{\times}_{\rmF_{v}}\rightarrow \rmR^{\square}_{v}$ and $\chi_{\overline{\beta}_{v}}\circ\mathrm{Art}_{\rmF_{v}}\vert_{\calO^{\times}_{\rmF_{v}}}: \calO^{\times}_{\rmF_{v}}\rightarrow \rmR^{\square}_{v}$ which  factor through $k^{\times}_{v}(l)$. Therefore there is a local $\calO$-algebra morphism $\calO[\Delta_{v}]\rightarrow \rmR^{\square}_{v}$. It is straightforward to check that this morphism is formally smooth of relative dimension $10$.  
\end{proof}

\begin{lemma}\label{tw-primes}
Suppose all the assumptions in Theorem \ref{main}.  For every integer 
\begin{equation*}
q\geq \dim_{k}\rmH^{1}(\rmG_{\rmF, \rmS}, \mathrm{ad}\overline{\rho}(1))
\end{equation*}
and every integer $n\geq 1$, there is Taylor-Wiles datum  $(\rmQ_{n}, (\overline{\alpha}_{v}, \overline{\beta}_{v}, \overline{\gamma_{v}}, \overline{\delta_{v}})_{v\in\rmQ_{n}})$ satisfying
\begin{enumerate}
\item $\vert\rmQ_{n}\vert=q$;
\item $ q_{v} \equiv 1\mod l^{n}$;
\item $\rmR^{\rmS}_{\calS_{\rmQ_{n}}}$ can be topologically generated over $\rmR^{\mathrm{loc}, \rmS}_{\calS}$ by
\begin{equation*}
g=2q-4[\rmF: \QQ]+\vert \rmS\vert-1.
\end{equation*}
elements.
\end{enumerate}
\end{lemma}
\begin{proof}
This is exactly \cite[Corollary 7.6.3]{BCGP}
\end{proof}

We now recall a few facts about the pro-$v$ Iwahori Hecke algebra and the Iwahori Hecke algbebra.
\begin{itemize}
\item Let $\TT^{\Iw_{1}}_{v}=\calO[\Iw_{1}(v)\backslash\GSp_{4}(\rmF_{v})/\Iw_{1}(v)]$ be the pro-$v$ Iwahori Hecke algebra. There is an injection $\calO[\rmT(\rmF_{v})/\rmT(\calO_{\rmF_{v}})_{1}]\hookrightarrow \TT^{\Iw_{1}}_{v}$ with $\rmT(\calO_{\rmF_{v}})_{1}=\ker(\rmT(\calO_{\rmF_{v}})\rightarrow \rmT(k_{v}))$ defined as in \cite[2.4.29]{BCGP}.

\item Let $\TT^{\Iw}_{v}=\calO[\Iw(v)\backslash\GSp_{4}(\rmF_{v})/\Iw(v)]$ be the Iwahori Hecke algebra. There is an injection $\calO[\rmT(\rmF_{v})/\rmT(\calO_{\rmF_{v}})]\hookrightarrow \TT^{\Iw}_{v}$ coming from the Bernstein presentation of $\TT^{\Iw}_{v}$. We will write $e_{\mathrm{Sph}}(v)$ for the Hecke operator $[\GSp_{4}(\calO_{\rmF_{v}})1\GSp_{4}(\calO_{\rmF_{v}})]$.

\item Let $\overline{\alpha}_{v}, \overline{\beta}_{v}$ be two elements of $k^{\times}$. Let $\fracm_{\overline{\alpha}_{v}, \overline{\beta}_{v}}$ be the kernel of the homomorphism $\calO[\rmT(\rmF_{v})/\rmT(\calO_{\rmF_{v}})_{1}]\rightarrow k$ induced by the character $\rmT(\rmF_{v})/\rmT(\calO_{\rmF_{v}})_{1}\rightarrow k^{\times}$ sending $\rmT(\calO_{\rmF_{v}})\mapsto 1$, $\beta_{v, 0}\mapsto 1$, $\beta_{v, 1}\mapsto \overline{\alpha}_{v}$ and $\beta_{v, 2}\mapsto \overline{\beta}_{v}$. Note $\fracm_{\overline{\alpha}_{v}, \overline{\beta}_{v}}$ can also be viewed as a maximal ideal of $\calO[\rmT(\rmF_{v})/\rmT(\calO_{\rmF_{v}})]\cong\calO[\rmX_{\ast}(\rmT)]$.
\end{itemize}

\begin{lemma}\label{Iwahori-Hecke}
Let $\rmM$ be a $\TT^{\Iw}_{v}\otimes_{\calO}k$-module which is finite dimensional over $k$.
\begin{enumerate}
\item Suppose that $e_{\mathrm{Sph}(v)}\rmM\neq 0$, and that there is a triple $\gamma_{0}, \gamma_{1}, \gamma_{2}$ with $\gamma^{2}_{0}\gamma_{1}\gamma_{2}=1$ such that $(\gamma_{1}-1)(\gamma_{2}-1)(\gamma_{1}-\gamma_{2})(\gamma_{1}\gamma_{2}-1)\neq 0$. Equivalently, writing $\overline{\alpha}_{v}=\gamma_{0}$ and $\overline{\beta}_{v}=\gamma_{0}\gamma_{1}$, suppose that $\overline{\alpha}^{\pm1}_{v}$ and $\overline{\beta}^{\pm1}_{v}$ are pairwise distinct. Suppose that the following operators are zero on the module $e_{\mathrm{Sph}(v)}\rmM$:
\begin{equation*}
\rmT_{v,0}-1, \rmT_{v, 2}-e_{1}(\gamma_{0}, \gamma_{1}, \gamma_{2}), \rmT_{v, 1}+2\rmT_{v,0}-e_{2}(\gamma_{0}, \gamma_{1}, \gamma_{2}).
\end{equation*}
Then for each $w\in\rmW$, the maximal ideal $\fracm_{w}=(x_{i}-(w\gamma)_{i})_{i=0,1,2}\in k[\rmX_{\ast}(\rmT)]$ is in the support of $\rmM$.
\item Conversely, suppose for each maximal ideal $\mathfrak{n}\in k[\rmX_{\ast}(\rmT)]^{\rmW}$ in the support of $\rmM$, the degree $4$ polynomial
\begin{equation*}
\sum e_{i}(x_{0}, x_{1}, x_{2})\rmX^{i}=(\rmX-x_{0})(\rmX-x_{0}x_{1})(\rmX-x_{0}x_{2})(\rmX-x_{0}x_{1}x_{2})
\end{equation*} 
in $k[\rmX_{\ast}(\rmT)]^{\rmW}[\rmX]$ has roots $\gamma_{0}, \gamma_{0}\gamma_{1}, \gamma_{0}\gamma_{1}\gamma_{2}$ modulo $\fracn$ satisfying $(\gamma_{1}-1)(\gamma_{2}-1)(\gamma_{1}-\gamma_{2})(\gamma_{1}\gamma_{2}-1)\neq 0$ and that $\gamma^{2}_{0}\gamma_{1}\gamma_{2}=1$. Equivalently, writing $\gamma_{0}=\overline{\alpha}_{v}$, $\gamma_{0}\gamma_{1}=\overline{\beta}_{v}$, assume that $\gamma^{2}_{0}\gamma_{1}\gamma_{2}=1$ and that
$\overline{\alpha}_{v}, \overline{\beta}_{v}, \overline{\alpha}^{-1}_{v}, \overline{\beta}^{-1}_{v}$ are pairwise distinct. Then $e_{\mathrm{Sph}(v)}\rmM\neq 0$. If furthermore, there is unique maximal ideal $\fracn\subset k[\rmX_{\ast}(\rmT)]^{\rmW}$ in the support of $\rmM$, then for each maximal ideal $\fracm\subset k[\rmX_{\ast}(\rmT)]$ in the support of $\rmM$, the maps
\begin{equation*}
\begin{aligned}
&k[\rmW]\otimes_{k}\rmM_{\fracm}\rightarrow \rmM,\phantom{p} w\otimes m \mapsto w\cdot m\\
&\rmM_{\fracm}\rightarrow e_{\mathrm{Sph}(v)}\rmM,\phantom{p} m\mapsto e_{\mathrm{Sph}}(v)m\\
\end{aligned}
\end{equation*}
are both isomorphisms.
\end{enumerate}
\end{lemma}
\begin{proof}
Part $(1)$ is \cite[Lemma 2.4.36]{BCGP} and part $(2)$ is \cite[Lemma 2.4.37]{BCGP}.
\end{proof}
Recall that an $\calO[\GSp_{4}(\rmF_{v})]$-module $\rmM$ is \emph{smooth} if every element of $\rmM$ is fixed by some open compact subgroup of $\GSp_{4}(\rmF_{v})$ and it is \emph{admissible} if it is smooth and for each open compact $\rmK$, $\rmM^{\rmK}$ is a finite $\calO$-modules.
The following modules are the admissible $\calO[\GSp_{4}(\rmF_{v})]$-modules we will consider.
\begin{equation*}
\begin{aligned}
&\rmM_{\emptyset}:=\Hom_{\calO}(\rmH^{d(\rmB)}(\mathrm{Sh}(\rmB, \rmK), \calL_{\xi}), \calO)_{\fracm}\\
&\rmM_{0}(\rmQ_{n}):=\rmH_{0}(\rmQ_{n})_{\fracm_{\rmQ_{n}}, \fracm^{\rmQ_{n}}_{\overline{\alpha}, \overline{\beta}}}\\ 
&\rmM_{1}(\rmQ_{n}):=\rmH_{1}(\rmQ_{n})_{\fracm_{\rmQ_{n}}, \fracm^{\rmQ_{n}}_{\overline{\alpha}, \overline{\beta}}}\\
\end{aligned}
\end{equation*}
where localizing at $\fracm^{\rmQ_{n}}_{\overline{\alpha}, \overline{\beta}}$ means localizing at each maxiamal ideal  $\fracm_{\overline{\alpha}_{v}, \overline{\beta}_{v}}$ for all $v\in\rmQ_{n}$. These modules are defined similarly as in \cite[page 410]{BCGP}.

\begin{lemma}\label{tw-module}
There is natural isomorphism 
\begin{equation*}
\rmM_{0}(\rmQ_{n})\xrightarrow{\sim}\rmM_{\emptyset}
\end{equation*}
which induces an isomorphism
\begin{equation*}
\rmM_{1}(\rmQ_{n})/\fraca_{\rmQ_{n}}\xrightarrow{\sim}\rmM_{\emptyset}.
\end{equation*}
\end{lemma}
\begin{proof}
For the first isomorphism, we use Nakyama's lemma and we are reduced to prove the composite of the following maps are isomorphisms
\begin{equation*}
\begin{aligned}
\rmH^{d(\rmB)}(\mathrm{Sh}(\rmB, \rmK_{0}(\rmQ)), \calL_{\xi}\otimes k)_{\fracm_{\rmQ}, \fracm^{\rmQ_{n}}_{\overline{\alpha}, \overline{\beta}}}&\xrightarrow{\alpha_{1}} \rmH^{d(\rmB)}(\mathrm{Sh}(\rmB, \rmK_{0}(\rmQ)), \calL_{\xi}\otimes k)_{\fracm_{\rmQ}};\\
\rmH^{d(\rmB)}(\mathrm{Sh}(\rmB, \rmK_{0}(\rmQ)), \calL_{\xi}\otimes k)_{\fracm_{\rmQ}}&\xrightarrow{\alpha_{2}}\rmH^{d(\rmB)}(\mathrm{Sh}(\rmB, \rmK), \calL_{\xi}\otimes k)_{\fracm_{\rmQ}}.\\
\rmH^{d(\rmB)}(\mathrm{Sh}(\rmB, \rmK_{0}(\rmQ)), \calL_{\xi}\otimes k)_{\fracm_{\rmQ}}&\xrightarrow{\alpha_{3}}\rmH^{d(\rmB)}(\mathrm{Sh}(\rmB, \rmK_{0}(\rmQ)), \calL_{\xi}\otimes k)_{\fracm}\\
\end{aligned}
\end{equation*}
is an isomorphism. Showing the last isomorphism $\alpha_{3}$ is easy and this results from local-global compatibility and Chebotarev density theorem. 

For each $v\in\rmQ$, let $\rmK_{v}=\GSp_{4}(\calO_{\rmF_{v}})$ and $\rmK^{\prime}_{v}=\Iw(v)$, the map $\alpha_{2}$ is induced from the Hecke operators $[\rmK_{v} 1\rmK^{\prime}_{v}]$. Let $\alpha^{\prime}_{2}$ be the map 
\begin{equation*}
\rmH^{d(\rmB)}(\mathrm{Sh}(\rmB, \rmK), \calL_{\xi}\otimes k)_{\fracm_{\rmQ}}\xrightarrow{\alpha^{\prime}_{2}}\rmH^{d(\rmB)}(\mathrm{Sh}(\rmB, \rmK_{0}(\rmQ)), \calL_{\xi}\otimes k)_{\fracm_{\rmQ}}
\end{equation*}
induced by $[\rmK^{\prime}_{v}1\rmK_{v}]$. Then since $[\rmK^{\prime}_{v}1\rmK_{v}][\rmK_{v}1\rmK^{\prime}_{v}]$ is by definition $e_{\mathrm{Sph}(v)}$. We see that the composite map $\alpha^{\prime}_{2}\circ\alpha_{2}\circ\alpha_{1}$ is an isomorphism by $(2)$ of Lemma \ref{Iwahori-Hecke}  and the definition of the Taylor-Wiles datum.  On the other hand, $\alpha_{2}\circ\alpha^{\prime}_{2}=[\rmK_{v}1\rmK^{\prime}_{v}][\rmK^{\prime}_{v}1\rmK_{v}]$ is $[\rmK_{v}:\rmK^{\prime}_{v}]$ which is non-zero in $k$. It follows that $\alpha^{\prime}_{2}$ is injective and hence $\alpha_{2}\circ\alpha_{1}$ is an isomorphism. The first part is proved. The second part follows from Lemma \ref{H-tw}. 
\end{proof}

It is well-known that $\bfT_{\fracm}[1/l]$ is reduced finite ring over $E_{\lambda}$. Since $\rmM_{\emptyset}$ is finite free $\calO$-module, $\bfT_{\fracm}$ is a reduced ring finite flat over $\calO$. Every closed point $x$ of $\Spec\phantom{.}\bfT_{\fracm}[1/l]$ gives rise to a representation $\pi_{x}$ of $\GSp_{4}(\mathbb{A}_{\rmF})$ such that
\begin{itemize}
\item[(a)] the associated representation $\rho_{\pi_{x}, \iota_{l}}$ has residual representation $\overline{\rho}$;
\item[(b)] $\pi^{\Sigma_{\rmB}}_{x}$ appears in $\rmH^{d(\rmB)}(\mathrm{Sh}(\rmB, \rmK), \calL_{\xi}\otimes \CC)$;
\item[(c)] the archimedean weights of $\pi_{x}$ equals to $\xi$.
\end{itemize}
The representation $\rho_{\pi_{x}, \iota_{l}}$ induces a continuous homomorphism
\begin{equation*}
\rho_{x}: \rmG_{\rmF}\rightarrow \GSp_{4}(\bfT_{x})
\end{equation*}
where $\bfT_{x}$ is the completion of $\bfT_{\fracm}[1/l]$ at $x$. By a theorem of Carayol, the representation
\begin{equation*}
\prod\limits_{x\in\Spec\phantom{.}\bfT_{\fracm}[1/l]}\rho_{x}: \rmG_{\rmF}\rightarrow \GSp_{4}(\prod\limits_{x}\bfT_{x})
\end{equation*}
can be conjugated to a representation
\begin{equation*}
\rho_{\fracm}: \rmG_{\rmF}\rightarrow \GSp_{4}(\bfT_{\fracm})
\end{equation*}
which lifts $\overline{\rho}$. 
\begin{lemma}
The Galois representation
\begin{equation*}
\rho_{\fracm}: \rmG_{\rmF}\rightarrow \GSp_{4}(\bfT_{\fracm})
\end{equation*}
satisfies the global deformation problem $\calS$. In particular there is a morphism 
\begin{equation*}
\varphi: \rmR^{\mathrm{univ}}_{\calS}\rightarrow \bfT_{\fracm}
\end{equation*}
which is in fact surjective.
\end{lemma}
\begin{proof}
Let $x\in \Spec\phantom{.}\bfT_{\fracm}[1/l]$. By $(\rmb)$ in the above discussion, we have $\pi_{x, v}$ is unramified for $v\not\in \rmS$ and hence $\rho_{\fracm}\vert_{\rmG_{\rmF_{v}}}$ is unramified. By $(\rmc)$, and $v\in \Sigma_{l}$, we have $\rho_{\fracm}\vert_{\rmG_{\rmF_{v}}}$ is crystalline of Fontiane-Laffaille weights of the desired range. By the assumptions in Theorem \ref{main}, we know $\pi_{x}$ is necessarily tempered and hence by local Jacquet-Langlands correspondence \cite{CG}, $\pi_{x ,v}$ is of type $\mathrm{II}_{\rma}$ for $v\in\Sigma_{\mathrm{lr}}$. Hence $\rho_{\fracm}\vert_{\rmG_{\rmF_{v}}}$ belongs to $\calD^{\mathrm{ram}}_{v}$ by local Langlands correspondence \cite{GT11}. Finally $\varphi$ clearly satisfies \cite[Proposition 3.4.4 (2,3)]{CHT} which implies $\varphi$ is surjective. 
\end{proof}

Let $(\rmQ_{n}, (\overline{\alpha}_{v}, \overline{\beta}_{v}, \overline{\gamma_{v}}, \overline{\delta_{v}}))$ be a Taylor-Wiles datum as in Theorem \ref{tw-primes}. We define $\bfT_{i, \fracm_{\rmQ_{n}}}$ by the image of $\TT^{\Sigma\cup\rmQ_{n}}$ in $\End_{\calO}(\rmM_{i}(\rmQ_{n}))$. By the same construction as above, we have a Galois representation
\begin{equation*}
\rho_{1, \fracm_{\rmQ_{n}}}: \rmG_{\rmF}\rightarrow \GSp_{4}(\bfT_{1, \fracm_{\rmQ_{n}}})
\end{equation*}
which satisfies the augmented deformation problem $\calS_{\rmQ_{n}}$. Therefore we have a morphism
\begin{equation*}
\varphi_{\rmQ_{n}}: \rmR^{\mathrm{univ}}_{\calS_{\rmQ_{n}}}\rightarrow \bfT_{1, \fracm_{\rmQ_{n}}}
\end{equation*}
which is surjective by the same reasoning as in the previous lemma.

Let $v$ be an element in $\rmQ_{n}$. Then $\rho_{1, \fracm_{\rmQ_{n}}}\vert_{\rmG_{\rmF_{v}}} $has the form in Lemma \ref{tw-lift}. Therefore there are characters 
$\chi_{\overline{\alpha}_{v}}\circ\mathrm{Art}_{\rmF_{v}}\vert_{\calO^{\times}_{\rmF_{v}}}:\calO^{\times}_{\rmF_{v}}\rightarrow \bfT^{\times}_{1, \fracm_{\rmQ_{n}}}$ and $\chi_{\overline{\beta}_{v}}\circ\mathrm{Art}_{\rmF_{v}}\vert_{\calO^{\times}_{\rmF_{v}}} :\calO^{\times}_{\rmF_{v}}\rightarrow \bfT^{\times}_{1, \fracm_{\rmQ_{n}}}$ which factor through $k^{\times}_{v}(l)$. This gives a morphism $\calO[\Delta_{v}]\rightarrow \bfT_{1, \fracm_{\rmQ_{n}}}$ and hence a morphism $\calO[\Delta_{\rmQ_{n}}]\rightarrow \bfT_{1, \fracm_{\rmQ_{n}}}$ which is compatible with the morphism $\calO[\Delta_{\rmQ_{n}}]\rightarrow \rmR^{\mathrm{univ}}_{\calS_{\rmQ_{n}}}$ constructed in the same way.
\begin{lemma}
There is a commutative diagram
\begin{equation*}
\begin{tikzcd}
 \rmR^{\mathrm{univ}}_{\calS_{\rmQ_{n}}}/\fraca_{\rmQ_{n}} \arrow[r, "\sim"]  \arrow[d, "\varphi_{\rmQ_{n}}/\fraca_{\rmQ_{n}}"]  &\rmR^{\mathrm{univ}}_{\calS} \arrow[d, "\varphi"] \\
 \bfT_{1, \fracm_{\rmQ_{n}}}/\fraca_{\rmQ_{n}}\arrow[r, "\sim"]        &    \bfT_{\fracm} 
\end{tikzcd}
\end{equation*}
\end{lemma}
\begin{proof}
This bottom isomorphism is the composite of the maps
\begin{equation*}
\bfT_{1, \fracm_{\rmQ_{n}}}/\fraca_{\rmQ_{n}}\xrightarrow{\sim} \bfT_{0, \fracm_{\rmQ_{n}}}\xrightarrow{\sim}\bfT_{\fracm}
\end{equation*}
which results from Lemma \ref{tw-module}. The upper isomorphism follows from Lemma \ref{tw-lift} and the above discussions. The commutativity of the diagram is clear.
\end{proof}

Now we collect some results in the previous discussion and get ready towards the patching construction.
\begin{itemize}
\item We have isomorphisms
\begin{equation*}
\rmR^{\univ}_{\calS}\widehat{\otimes}\calT\xrightarrow{\sim} \rmR^{\rmS}_{\calS},\phantom{aa} \rmR^{\univ}_{\calS_{\rmQ_{n}}}\widehat{\otimes}\calT\xrightarrow{\sim} \rmR^{\rmS}_{\calS_{\rmQ_{n}}}.
\end{equation*}
Thus $\rmR^{\univ}_{\calS}$ is an algebra over $\rmR^{\rmS, \mathrm{loc}}_{\calS}$.
\item We define rings
\begin{equation*}
S_{\infty}:=\calO[\calT][[\rmY_{1}, \cdots, \rmY_{2q}]],\phantom{aa}\rmR_{\infty}=\rmR^{\rmS, \mathrm{loc}}_{\calS}[[\rmZ_{1},\cdots, \rmZ_{g}]]
\end{equation*}
and $\fraca_{\infty}\subset S_{\infty}$ be the augmentation ideal where the numbers $g$ and $q$ are given in Lemma \ref{tw-primes}. 
\item We have the commutative diagram 
\begin{equation*}
\begin{tikzcd}
 \rmR^{\mathrm{univ}}_{\calS_{\rmQ_{n}}}/\fraca_{\rmQ_{n}} \arrow[r, "\sim"]  \arrow[d, "\varphi_{\rmQ_{n}}/\fraca_{\rmQ_{n}}"]  &\rmR^{\mathrm{univ}}_{\calS} \arrow[d, "\varphi"] \\
 \bfT_{1, \fracm_{\rmQ_{n}}}/\fraca_{\rmQ_{n}}\arrow[r, "\sim"]        &    \bfT_{\fracm} 
\end{tikzcd}
\end{equation*}
We have a $\bfT_{1, \fracm_{\rmQ_{n}}}$-module $\rmM_{1}(\rmQ_{n})$ and a $\bfT_{\fracm}$-module $\rmM_{\emptyset}$ which satisfy
\begin{equation*}
\rmM_{1}(\rmQ_{n})/\fraca_{\rmQ_{n}}\xrightarrow{\sim}\rmM_{\emptyset}.
\end{equation*}
\end{itemize}
Resulting from the patching lemma given by \cite[Lemma 6.10]{Tho12}, we have the following.
\begin{itemize}
\item There is a homomorphism $S_{\infty}\rightarrow \rmR_{\infty}$ of rings over $\calO$ such that we have an isomorphism 
\begin{equation*}
\rmR_{\infty}/\fraca_{\infty}\cong \rmR^{\mathrm{univ}}_{\calS}
\end{equation*}
over $\rmR^{\rmS, \mathrm{loc}}_{\calS}$. 
\item There is an $\rmR_{\infty}$-module $\rmM_{\infty}$ and an isomorphism $\rmM_{\infty}/\fraca_{\infty}=\rmM_{\emptyset}$ of $\rmR^{\mathrm{univ}}_{\calS}$-modules. Moreover $\rmM_{\infty}$ is a finite free $S_{\infty}$-module.
\end{itemize}
Finally, we are ready to prove our first main result.

\begin{theorem}
Under the above setup, suppose the following assumptions hold.
\begin{itemize}
\item[(D0)] $l\geq 5$ and is unramified in $\rmF$; 
\item[(D1)] The weight $(k_{v}, l_{v})_{v\mid\infty}$ of $\pi$ satisfies $l_{v}+k_{v}\leq l+1$;
\item[(D2)] The image $\overline{\rho}(\rmG_{\rmF})$ contains $\GSp_{4}(\FF_{l})$;
\item[(D3)] $\overline{\rho}$ is rigid for $(\Sigma_{\mathrm{min}}, \Sigma_{\mathrm{lr}})$;
\item[(D4)] Suppose that $\rmB$ is indefinite. For every finite set $\Sigma^{\prime}$ of nonarchimedean places of $\rmF$ containing $\Sigma$, and every open compact subgroup $\rmK^{\prime}\subset \rmK$ satisfying $\rmK^{\prime}_{v}=\rmK_{v}$ for $v\not\in \Sigma^{\prime}$, we have 
\begin{equation*}
\rmH^{d}(\mathrm{Sh}(\rmB, \rmK^{\prime}), \calL_{\xi}\otimes_{\calO}k)_{\fracm^{\prime}}=0
\end{equation*}
for $d\neq 3$ and $\fracm^{\prime}=\fracm\cap \TT^{\Sigma^{\prime}}$.
\end{itemize}
Let $\mathbf{T}$ be the image of $\TT^{\Sigma}$ in $\mathrm{End}_{\calO}(\rmH^{d(\rmB)}(\mathrm{Sh}(\rmB, \rmK), \calL_{\xi}))$ and suppose that $\mathbf{T}_{\fracm}$ is non-zero. Then the following results hold.
\begin{enumerate}
\item There is an isomorphism of complete intersection rings: $\rmR^{\univ}_{\calS}\cong\mathbf{T}_{\fracm}$.
\item The cohomology group $\rmH^{d(\rmB)}(\mathrm{Sh}(\rmB, \rmK), \calL_{\xi})_{\fracm}$ is finite free module over $\mathbf{T}_{\fracm}$ and hence $\rmH^{d(\rmB)}(\mathrm{Sh}(\rmB, \rmK), \calL_{\xi})_{\fracm}$ is also finite free over $\rmR^{\univ}_{\calS}$.
\end{enumerate}
\end{theorem}

\begin{proof}
From the above discussions, we have 
\begin{equation*}
\mathrm{depth}_{\rmR_{\infty}}\phantom{.}\rmM_{\infty}\geq \dim S_{\infty}= 2q+11\vert\rmS\vert.
\end{equation*}
By Proposition \ref{min-dim}, we have $\dim_{\calO} \rmR^{\min}_{v}=10$ for each $v\in\Sigma_{\mathrm{min}}$; By Proposition \ref{mix-dim}, $\dim_{\calO} \rmR^{\mathrm{mix}}_{v}=10$ for each $v\in\Sigma_{\mathrm{mix}}$; By Proposition \ref{FL-dim}, we have $\dim_{\calO} \rmR^{\mathrm{FL}}_{v}=10+4[\rmF_{v}: \QQ_{l}]$ for each $v\in\Sigma_{l}$. Therefore we have
\begin{equation*}
\dim \rmR^{\mathrm{S}, \mathrm{loc}}_{\calS}=10\vert \rmS\vert + 4[\rmF: \QQ]
\end{equation*}
and hence $\rmR_{\infty}$ is a regular local ring of dimension
\begin{equation*}
\begin{aligned}
\dim \rmR_{\infty} &= g+1+10\vert \rmS\vert + 4[\rmF: \QQ]\\
&= 2q-4[\rmF: \QQ]+\vert \rmS\vert+10\vert\rmS\vert+4[\rmF: \QQ]\\
&=2q+11\vert\rmS\vert.
\end{aligned}
\end{equation*}
Note we now have 
\begin{equation*}
\dim S_{\infty}=2q+11\vert\rmS\vert\leq\mathrm{depth}_{\rmR_{\infty}}\phantom{.}\rmM_{\infty}\leq \dim \rmR_{\infty}=2q+11\vert\rmS\vert.
\end{equation*}
Thus by the Auslander-Buchsbaum theorem $\rmM_{\infty}$ is finite free over $\rmR_{\infty}$-module. Hence $\rmM_{\emptyset}$ is a finite free $\rmR^{\mathrm{univ}}_{\calS}$-module. It also follows that the morphism $\varphi: \rmR^{\univ}_{\calS}\rightarrow \bfT_{\fracm}$ is both surjective and injective. We are done.
\end{proof}

\section{Rigidity of automorphic Galois representations}

Let $\pi$ be a cuspidal automorphic representation of $\GSp_{4}(\mathbb{A}_{\rmF})$ of general type with trivial central character. We recall a result of Mok and Sorensen on the existence of Galois representations attached to $\pi$. We assume that $\pi$ has weight $(k_{v}, l_{v})_{v\mid \infty}$ and trivial central character.
\begin{theorem}\label{Galois}
Suppose $\pi$ is a cuspidal automorphic representation of $\GSp_{4}(\mathbb{A}_{\rmF})$ of weight $(k_{v}, l_{v})_{v\mid\infty}$ where $k_{v} \geq l_{2}\geq 3$ and $k_{v}\equiv l_{v}\mod 2$ for all $v\mid \infty$. Suppose that $\pi$ has trivial central character. Then there is a continuous semisimple representation $\rho_{\pi, \iota_{l}}: \rmG_{\rmF}\rightarrow \GSp_{4}(\overline{\QQ}_{l})$ satisfying the following.
\begin{enumerate}
\item $\nu\circ\rho_{\pi, \iota_{l}}=\epsilon^{-3}_{l}$;
\item For each finite place $v$, we have
\begin{equation*}
\mathrm{WD}(\rho_{\pi,\iota_{l}}\vert_{\rmG_{\rmF_{v}}})^{\mathrm{F}-\mathrm{ss}}\cong \mathrm{rec}_{\mathrm{GT}}(\pi_{v}\otimes\vert\nu\vert^{-3/2})^{\mathrm{ss}};
\end{equation*}
\item If $v\mid l$, then $\rho_{\pi,\iota_{l}}\vert_{\rmG_{\rmF_{v}}}$ is de Rham with Hodge-Tate weights 
\begin{equation*}
(2-\frac{k_{v}+l_{v}}{2}, -\frac{k_{v}-l_{v}}{2}, 1+\frac{k_{v}-l_{v}}{2}, -1+\frac{k_{v}+l_{v}}{2}).
\end{equation*}
Moreover, if $\pi$ is unramfied at $v$, then $\rho_{\pi,\iota_{l}}\vert_{\rmG_{\rmF_{v}}}$ is crystalline.
\item If $\rho_{\pi, \iota_{l}}$ is irreducible, then for all finite place $v$, $\rho_{\pi, \iota_{l}}\vert_{G_{\rmF_{v}}}$ is pure.
\end{enumerate}
\end{theorem}
\begin{remark}
The construction of the Galois representation in the case $\rmF=\QQ$ is constructed in \cite{La-Siegel, La-SiegelII, Wei-Galois, Tay-Siegel} using the Langlands-Kottwitz method. The statements about local-global compatibilities are mostly proved in \cite{Sor-local-global} and completed in \cite{Mok-GSp}. These works also construct the Galois representations in the totally real field case by using the strong transfer of cuspidal automorphic representation from $\GSp_{4}$ to $\GL_{4}$. The purity statement in the last point follows from the work of Carianni \cite{Carai}. 
\end{remark}

\begin{definition}\label{strong-field}
Let $\pi$ be a cuspidal automorphic representation of $\GSp_{4}(\mathbb{A}_{\rmF})$. We say a number field $E\subset \CC$ is a strong coefficient field of $\pi$ if for every prime $\lambda$ of $E$ there exits a continuous homomorphism 
\begin{equation*}
\rho_{\pi,\lambda}: G_{\rmF}\rightarrow \GSp_{4}(E_{\lambda})
\end{equation*}
up to conjugation, such that for every isomorphism $\iota_{l}:\CC\xrightarrow{\sim}\overline{\QQ}_{l}$ inducing the prime $\lambda$, the representations $\rho_{\pi,\lambda}\otimes_{E_{\lambda}}\overline{\QQ}_{l}$ and $\rho_{\pi,\iota}$ are conjugate.

\end{definition}

We fix such a strong coefficient field $E$ for $\pi$.  We assume that $\rho_{\pi, \iota_{l}}$ is absolutely irreducible and so is $\rho_{\pi,\lambda}$. Therefore we can consider the residual Galois representation
\begin{equation*}
\overline{\rho}_{\pi, \lambda}: G_{\rmF}\rightarrow \GSp_{4}(k).
\end{equation*}
Let $\Sigma$ be the smallest set containing $\Sigma_{\mathrm{bad}}$ such that $\pi$ is unramified for every non-archimedean place $v$ of $\rmF$
not in it. We will establish that $\overline{\rho}_{\pi, \lambda}$  is rigid under suitable hypothesis. Before that we need the following lemma.
\begin{lemma}\label{bernstein}
Let $v$ be a non-archimedean place of $\rmF$ and suppose that $l\nmid\prod^{4}_{i=1} q^{i}_{v}-1$. Let $A\in \CNL_{\calO}$ and $\rho_{1}$ and $\rho_{2}$ are two liftings of $\overline{\rho}_{\pi,\lambda}$ valued in $A$. If $\pi_{1, v}$ and $\pi_{2, v}$ are two irreducible admissible representation of $\GSp_{4}(\rmF_{v})$ giving the the representations $\rho_{1, v}\otimes\overline{\QQ}_{l}$ and $\rho_{2, v}\otimes\overline{\QQ}_{l}$ via the local Langlands correspondence. Then $\pi_{1, v}$ and $\pi_{2,v}$ are in the same Bernstein component. 
\end{lemma}
\begin{proof}
Let $\Pi_{1, v}$ and $\Pi_{2, v}$ be the transfer of $\pi_{1, v}$ and $\pi_{2, v}$ as representations of $\GL_{4}(\rmF_{v})$. Then $\Pi_{1, v}$ and $\Pi_{2, v}$ lie on the same Bernstein component by \cite[Corollary 3.4.10]{LTXZZa}. It then follows that $\pi_{1, v}$ and $\pi_{2, v}$ are in the same Bernstein component as well.
\end{proof}

\begin{theorem}
Suppose $\pi$ is a cuspidal automorphic representation of $\GSp_{4}(\mathbb{A}_{\rmF})$ with trivial central character of weight $(k_{v}, l_{v})_{v\mid\infty}$ where $k_{v}, l_{v}\geq 3$ and $k_{v}\equiv l_{v}\mod 2$ for all $v\mid \infty$. We assume that the image $\overline{\rho}_{\pi, \lambda}(G_{\rmF})$ contains $\GSp_{4}(\FF_{l})$ for sufficiently large $l$. Then the representation $\overline{\rho}_{\pi, \lambda}$ is rigid for the pair $(\Sigma,\emptyset)$ for infinitely many choices of $l$.
\end{theorem}
\begin{proof}
To establish that $\overline{\rho}_{\pi, \lambda}$ is rigid, we need to establish the four conditions in Definition \ref{rigid}. We begin by choose $l$ such that 
\begin{enumerate}
\item $\Sigma_{l}\cap\Sigma=\emptyset$;
\item $l\geq k_{v}+l_{v}-1$ for all $v\mid \infty$;
\item $l\geq q^{4}_{v}$ for all $v\in\Sigma$.
\end{enumerate}
The all four conditions but $(1)$ is clear for $\overline{\rho}_{\pi, \lambda}$ by Theorem \ref{Galois}. We need to verify $(1)$ for it.  Let $\calD_{\Sigma}$ be a collection $\{\calD_{v}, v\in\Sigma\}$ be a collection of local deformation problems, each of which is an irreducible component of $\rmR^{\square}_{v}$ for $\overline{\rho}_{\pi, \lambda}\vert_{G_{\rmF_{v}}}$ for $v\in\Sigma$. Consider the following global deformation problem
\begin{equation*}
\calS(\calD_{\Sigma})=(\overline{\rho}_{{\pi},\lambda}, \epsilon^{-3}_{l}, \Sigma\cup\Sigma_{l}, \{\calD_{v}\}_{v\in\Sigma\cup\Sigma_{l}})
\end{equation*}
where $\calD_{v}\in\calD_{\Sigma}$ for $v\in\Sigma$ and $\calD_{v}=\calD^{\mathrm{FL}}_{v}$ for $v\in\Sigma_{l}$. By \cite[Corollary 7.5.1]{GG-companion}, the universal deformation ring $\rmR^{\mathrm{univ}}_{\calS(\calD_{\Sigma})}$ is finite over $\calO$ of rank at least $1$. In particular, we know $\rmR^{\mathrm{univ}}_{\calS(\calD_{\Sigma})}[1/l]$ is non-zero.  A $\overline{\QQ}_{l}$ point of $\Spec\phantom{.}\rmR^{\mathrm{univ}}_{\calS(\calD_{\Sigma})}[1/l]$ gives rise to cuspidal representation $\pi(\calD_{\Sigma})$ of $\GSp_{4}(\mathbb{A}_{\rmF})$ such that
\begin{itemize}
\item $\pi(\calD_{\Sigma})$ is unramified away from $\Sigma$;
\item There is an open compact subgroup $\rmK_{v}=\rmK_{v}(\pi_{v})$ depending only on $\pi_{v}$ such that $\pi(\calD_{\Sigma})^{\rmK_{v}}_{v}\neq 0$;
\item The weights $(k_{v}, l_{v})_{v\mid\infty}$ are cohomological in the Fontaine-Laffaille range: $k_{v}\geq l_{v}\geq 3$ and $k_{v}+l_{v}\leq l+1$.
\item The representations $\rho_{\pi(\calD_{\Sigma}),\iota_{l}}$ and $\rho_{\pi, \lambda}$ are residually isomorphic. 
\end{itemize}
Note the second property is a direct consequence of Lemma \ref{bernstein} and the fact that representations on the same Bernstein component has the same level. It is clear that there are only finitely many cuspidal representations of $\GSp_{4}(\mathbb{A}_{\rmF})$ of general type up to isomorphism satisfying the first three properties. Finally by strong multiplicity one for representations of general type, we can choose $l$ large enough so that $\pi$ itself is the unique cuspidal representations of $\GSp_{4}(\mathbb{A}_{\rmF})$ of general type satisfying all $4$ properties.

Suppose we have two different collections $\calD_{\Sigma}$ and $\calD^{\bullet}_{\Sigma}$ of local deformation problems. We need to show that $\pi(\calD_{\Sigma})$ and $\pi(\calD^{\bullet}_{\Sigma})$ are not isomorphic. For this, we will show that $\pi_{v}$ for each $v\in\Sigma$ gives rise to a smooth point on $\Spec\phantom{.}\rmR^{\square}_{v}[1/l]$ and hence can not lie on two different irreducible components. This follows from \cite[Lemma 7.1.3]{BCGP} in the usual manner as $\pi_{v}$ gives rise to a pure local Galois representation. It follows that if $l$ is large enough as above, then $\calD_{v}=\calD^{\mathrm{min}}_{v}$ for all $v\in\Sigma$ and hence $(1)$ of Definition \ref{rigid} is verified. Thus $\overline{\rho}_{\pi, \lambda}$ is rigid. 
\end{proof}

\end{document}